 \documentclass[12pt]{article}
\usepackage{amssymb, amsthm, amsmath, amsfonts}
\usepackage{url} 
\usepackage{enumitem} 
\usepackage[colorlinks,
            linkcolor=black,
            anchorcolor=black,
            citecolor=black,
            ]{hyperref} 

\usepackage{lineno} 

\usepackage{ mathtools} 
\mathtoolsset{showonlyrefs}

\usepackage{setspace}
\newtheorem{thm}{Theorem}[section]
\newtheorem{cor}[thm]{Corollary}
\newtheorem{lem}[thm]{Lemma}
\newtheorem{prop}[thm]{Proposition}
\theoremstyle{definition}
\newtheorem{defn}[thm]{Definition}
\theoremstyle{remark}
\newtheorem{rem}[thm]{Remark}
\numberwithin{equation}{section}
\newcommand{\e}{{\rm e}}
\newcommand{\N}{\mathbb{N}}
\newcommand{\R}{\mathbb{R}}
\newcommand{\bigO}{\ensuremath{\mathcal{O}}}

\newcommand{\Exp}[2][]{\mathbb{E}_{#1}\left[ #2 \right]}
\newcommand{\Expcond}[2]{\mathbb{E}\left[ \left. #1 ~\right|~ #2 \right]}
\newcommand{\Prob}[2][]{\mathbb{P}_{#1}\left( #2 \right)}
\newcommand{\Probcond}[2]{\mathbb{P}\left( \left. #1 ~\right|~ #2 \right)}
\newcommand{\ind}[1]{\mathbf{1}_{\left\{ #1 \right\} } }
\newcommand{\F}{\mathcal{F}} 

\newcommand{\cvL}{ \quad \text{in } L^2(\mathbb{P}) } 
\newcommand{\lime}{\lim_{\epsilon \to 0}}
\newcommand{\LP}{ L^2(\mathbb P) }

\newcommand{\D}{\mathbb{D}}
\newcommand{\B}{\mathcal{B}}

\newcommand{\Sd}{S^{\downarrow}}
\newcommand{\s}{\mathbf{s}}
\newcommand{\X}{\mathbf{X}}

\newcommand{\nubar}{g} 
\newcommand{\Nld}{N} 
\newcommand{\Ald}{A} 
\newcommand{\markd}{(x,\s)\in [0,1]\times \Sd} 

\begin{document}
\title{On the number of large triangles in the Brownian triangulation and fragmentation processes}
\author{Quan Shi\thanks{Institut f\"ur Mathematik, Universit\"at Z\"urich,
Winterthurerstrasse 190, CH-8057 Z\"urich, Switzerland \protect\\ email: quan.shi@math.uzh.ch} 
\thanks{The author is deeply indebted to Jean Bertoin for suggesting this research and for many fruitful discussions. 
 The author thanks Igor Kortchemski for reading this paper carefully and for many useful comments. The author also thanks two anonymous referees whose suggestions and remarks helped to improve this paper.
This work is supported by the Swiss National Science Foundation 200021\_144325/1.}
}
\date{\today}
\maketitle
\begin{abstract}  The Brownian triangulation is a random compact subset of the unit disk introduced by Aldous. For $\epsilon>0$, let $N(\epsilon)$ be the number of triangles whose sizes (measured in different ways) are greater than $\epsilon$ in the Brownian triangulation. We determine the asymptotic behavior of $N(\epsilon)$ as $\epsilon \to 0$.

To obtain this result, a novel concept of ``large'' dislocations in fragmentations has been proposed. 
We develop an approach to study the number of large dislocations which is widely applicable to general self-similar fragmentation processes. This technique enables us to study $N(\epsilon)$ because of a bijection between the triangles in the Brownian triangulation and the dislocations of a certain self-similar fragmentation process.

Our method also provides a new way to obtain the law of the length of the longest chord in the Brownian triangulation. We further extend our results to the more general class of geodesic stable laminations introduced by Kortchemski.

\end{abstract}
{\bf 2010 Mathematics Subject Classiﬁcation:} 60F25, 60G18.\\
{\bf Keywords:} Brownian triangulation; Self-similar fragmentation; Geodesic stable lamination.

\newpage
\section{Introduction}\label{section:intro}
For $n\in \N$, let $P_n$ be the polygon formed by the $n$ roots of unity. A {\it triangulation of $P_n$} is the union of its sides and $(n-3)$ non-crossing (except at the endpoints) diagonals, thus dividing $P_n$ into $(n-2)$ triangles. A {\it uniform triangulation $\mathcal{T}_n$ of $P_n$} is a triangulation chosen uniformly at random from the set of all the different triangulations of $P_n$. In \cite{aldous1994triangulating}, Aldous regarded $\mathcal{T}_n$ as a random compact subset of the closed unit disk $\D\subset \R^2$ and showed that, as $n$ tends to infinity, $\mathcal{T}_n$ converges to a limit random compact set $\mathcal B$ in distribution for the Hausdorff metric.  Figure \ref{fig_BT} shows a sample of $\mathcal B$. 
\begin{figure}[h]\label{fig_BT}
\centering
\includegraphics[width=0.35\textwidth]{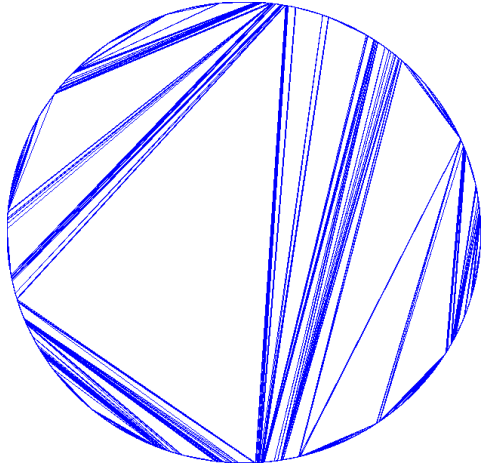}
\caption{ {\bf A sample of the Brownian triangulation.}}
\end{figure}

It turns out that $\B$ is a random {\it triangulation of the disk}, in the sense that it is a random closed subset of $\D$, whose complement $\D \backslash \mathcal B$ is a union of open triangles with vertices on the unit circle $\partial \D$. Aldous called $\B$ the {\it Brownian triangulation} since it can be encoded by a normalized Brownian excursion $e=(e_s, s\in [0,1])$ as follows. Parameterize $\partial \D$ by $({\rm e}^{i2\pi s}, s\in [0,1))$, and write $[{\rm e}^{i2\pi s}, {\rm e}^{i2\pi t}]$ for the chord connecting ${\rm e}^{i2\pi s}, {\rm e}^{i2\pi t} \in \partial \D$. Then almost surely 
\begin{equation}\mathcal B= \bigcup_{s \overset{e}{\sim} t, s,t\in[0,1)}[{\rm e}^{i2\pi s}, {\rm e}^{i2\pi t}], \end{equation}
 where $s \overset{e}{\sim} t$ if and only if $e(s)=e(t)=\min_{r\in [s\wedge t, s\vee t]}e(r)$. See \cite{ aldous1994triangulating} for details. 

The Brownian triangulation draws our attention because of its importance in many aspects. The Brownian triangulation is universal, as it is the limit of various random non-crossing configurations (collections of non-crossing diagonals) of $P_n$ \cite{curien2012random}. The Brownian triangulation is also closely related to the Brownian Continuum random tree (CRT) \cite{aldous1994recursive, aldous1994triangulating} and the Brownian map \cite{legall2008scaling}. Finally, the Brownian triangulation has provoked the study of other random triangulations, such as random recursive triangulations \cite{curienlegall2011random} and the Markovian hyperbolic triangulation \cite{curien2013hyperbolic}. 

By definition, a {\it triangle}, or {\it face}, of $\B$ is a connected component of $\D \backslash \B$. 
In the present work, we are mainly interested in the number of ``large'' triangles in $\B$. Clearly there are various ways of measuring the size of a triangle. Here we are concerned with two different ways, specifically the length of the shortest edge and the area. 

Let us now present a special case of our results. Recall that $e$ is the normalized Brownian excursion that encodes $\mathcal B$. We define a family of random open sets
\begin{equation}\label{eq:Theta}
 \Theta_e(t) := \left\{s\in (0,1): e(s)>t \right\}, \quad t\geq 0.
\end{equation}
For every $t\geq 0$, write $\Theta_e(t) = \bigcup_{i\in \N}I_i(t)$, where $( I_i(t), i\in \N)$ are the connected components of $\Theta_e(t)$. Hence $( I_i(t), i\in \N)$ are disjoint open intervals, possibly empty. We denote the length of an interval $I$ by $|I|$. 

\begin{thm}\label{thm:Brownian}
\begin{enumerate}
\item For every $\epsilon>0$, let $N'(\epsilon)$ be the number of triangles in $\mathcal B$ whose edges have lengths greater than $\epsilon$. There is 
\begin{equation}
 \lime \epsilon N'(\epsilon) = 2 \cvL.
\end{equation}
\item Let $N''(\epsilon)$ be the number of triangles in $\mathcal B$ whose Euclidean area is larger than $\epsilon>0$. There is 
\begin{equation}
 \lime \epsilon^{\frac{1}{2}} N''(\epsilon) = 4 \int_{0}^{\infty} \sum_{i=1}^{\infty}\sin( \pi |I_i(s)|) ds \cvL.
\end{equation}
\end{enumerate}
\end{thm}

We note that in the first case the limit is a constant, while in the second case the limit is a random variable. It turns out that this surprising phenomenon is an instance of a general phase transition revealed in Theorem \ref{thm:large} below. To justify that the random variable $ \int_{0}^{\infty} \sum_{i=1}^{\infty}\sin( \pi |I_i(s)|) ds$ is indeed square integrable, 
let us compare it with the {\it Brownian excursion area} $\mathcal A_e$,
\begin{equation}
\mathcal A_e:=\int_{0}^{\infty} \sum_{i=1}^{\infty}|I_i(t)| dt =\int_0^1 e(s)ds.
\end{equation}
It is known that $\Exp{\mathcal A_e^k}<\infty$ for every $k\in \N$, see \cite{janson2007area}. Noticing that
\begin{equation} 
\int_{0}^{\infty} \sum_{i=1}^{\infty}\sin( \pi |I_i(s)|) ds\leq \pi \mathcal A_e,
\end{equation}
we see that it is indeed square integrable. 

It has been proved in \cite{aldous1994triangulating} that for $0\leq x_1\leq x_2\leq x_3\leq 1$, the expected value of the number of the triangles whose vertices are at position $(e^{i2\pi x_1},e^{i2\pi x_2},e^{i2\pi x_3})$ has density
\begin{equation}
\frac{1}{4\pi} (x_2-x_1)^{-\frac{3}{2}}(x_3-x_2)^{-\frac{3}{2}}(1+x_1-x_3)^{-\frac{3}{2}} dx_1dx_2dx_3, \quad 0\leq x_1\leq x_2\leq x_3\leq 1.
\end{equation}
By integrating the density function, we may deduce that
\begin{linenomath}\begin{align}
\lim_{\epsilon \to 0} \epsilon \Exp{N'(\epsilon)} &= 2, \\
\lim_{\epsilon \to 0} \epsilon^{\frac{1}{2}} \Exp{N''(\epsilon)} &= \frac{\sqrt{2}\pi}{2} J_1(\frac{\pi}{2}),
\end{align}\end{linenomath}
where $J_1$ is the Bessel function of the first kind, with $J_1(\frac{\pi}{2})\approx 0.5668$. However, this result is weaker than our convergence in $L^2(\mathbb P)$ for random variables.

Theorem \ref{thm:Brownian} is proved in Section \ref{section:Brownian}. Our approach to tackle this problem is through a connection with fragmentation processes. It has been proved by Bertoin \cite{bertoin2002self} that the process $\Theta_e$ given by \eqref{eq:Theta} is an example of a {\it self-similar interval-partition fragmentation with index $-\frac{1}{2}$} (see Section \ref{section:background} for background). 
Roughly speaking, The process $\Theta_e$ describes how the interval $(0,1)$ splits into smaller intervals as time grows. 
For $s> t$, $\Theta_e (s)$ is obtained from $\Theta_e (t)$ by breaking randomly into pieces each component of $\Theta_e(t)$ according to a law that only depends on the length of this component, and independently of the others. We will specify this law in Section \ref{section:Brownian}. An interval splitting event is called a {\it dislocation}. We point out that in each dislocation of $\Theta_e$, an interval $I\subset (0,1)$ of length $|I|$ must split into two pieces $(I_1, I_2)$ with $|I_1|+|I_2| = |I|$. Such a dislocation is marked by $(|I|,(|I_1|/|I|,|I_2|/|I|)) \in (0,1]\times \Delta$, where  
$\Delta:=\{(s_1,s_2)\in [0,1]^2: s_1+s_2= 1, s_1\geq s_2 \}$.

\begin{figure}[h]
\begin{minipage}[t]{0.45\linewidth}
\centering
\includegraphics[width=1\textwidth]{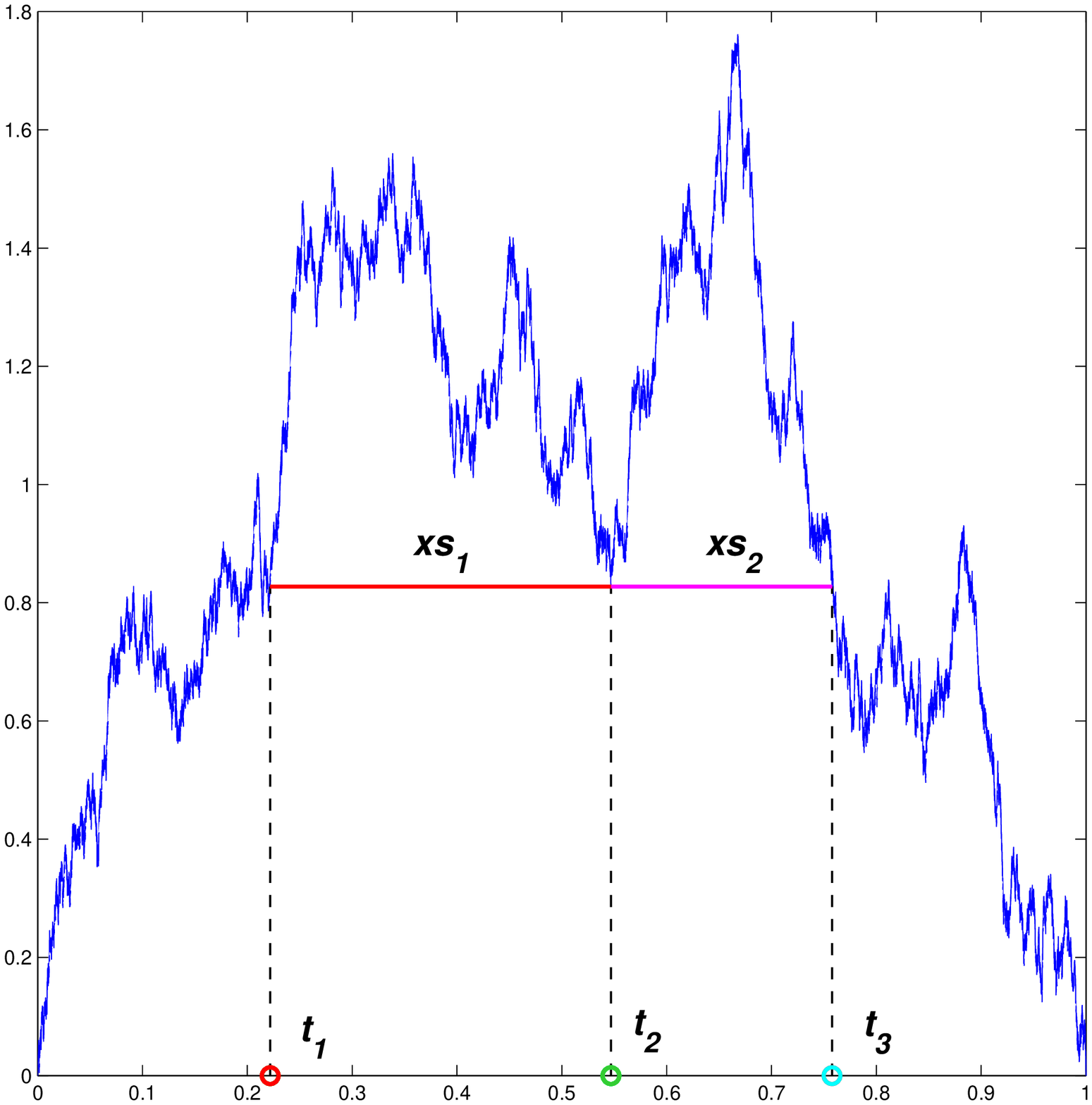}
\end{minipage}
\begin{minipage}[t]{0.45\linewidth}
\centering
\includegraphics[width=1\textwidth]{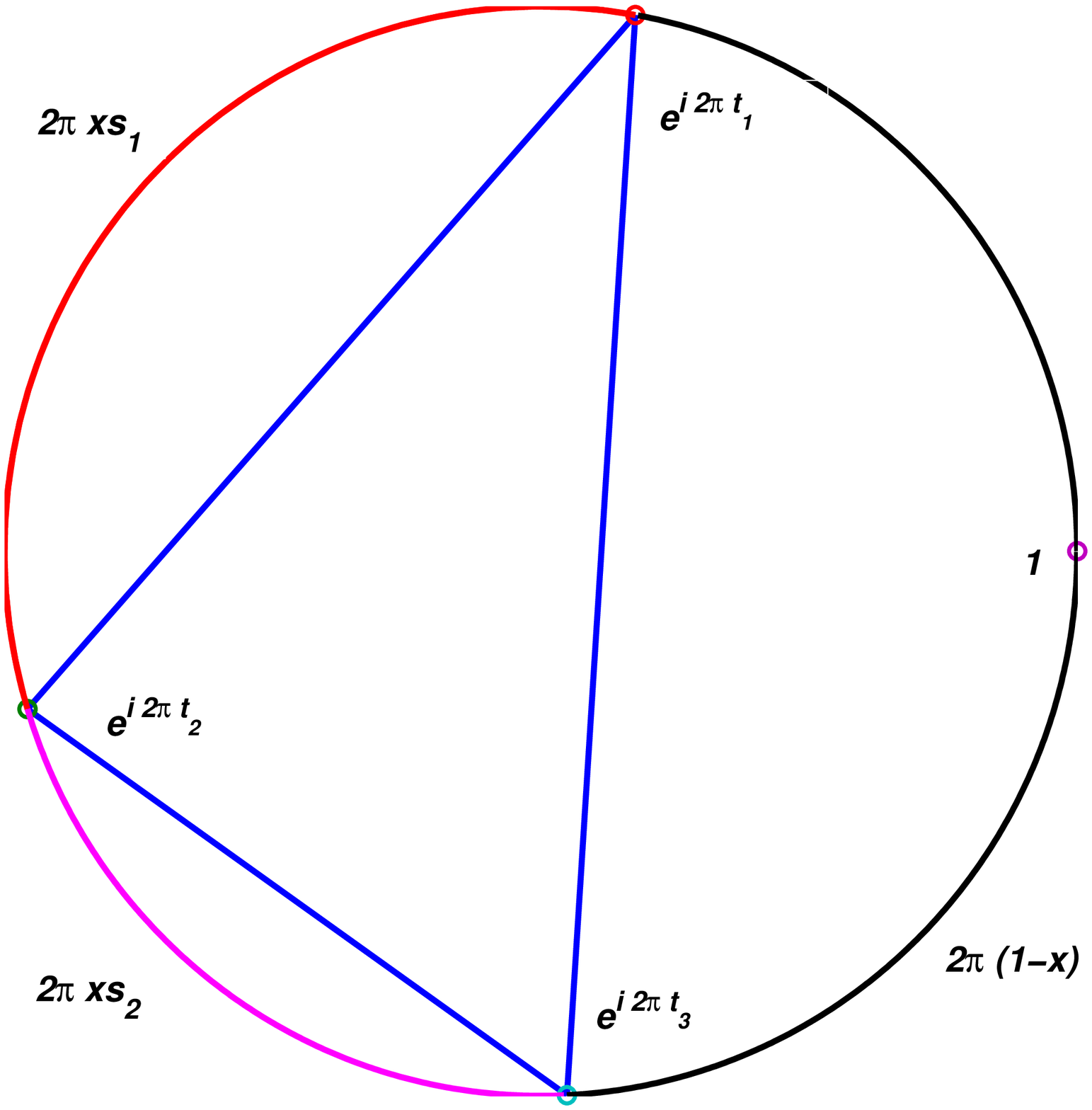}
\end{minipage}
\caption{ \label{fig_dual}\textbf{The correspondence between dislocations and triangles.} The local minimum $t_2$ of the Brownian excursion $e$ on the left induces a dislocation of $\Theta_e$, which corresponds to the triangle in $\B$ on the right.
In this dislocation the interval $(t_1, t_3)$ of length $x= t_3-t_1$ produces two intervals $(t_1, t_2)$ and $(t_2, t_3)$. Set $s_1 = \max (t_2- t_1, t_3-t_2)/x$ and $s_2 = 1- s_1$, then this dislocation is marked by $(x, (s_1,s_2))$. Since $t_1\overset{e}{\sim}t_2 \overset{e}{\sim} t_3$, the chords $[\e^{i2 \pi t_1}, \e^{i2 \pi t_2}]$, $[\e^{i2 \pi t_2}, \e^{i2 \pi t_3}]$ and $[\e^{i2 \pi t_3}, \e^{i2 \pi t_1}]$ are included in $\B$, and they form a triangle. Hence this dislocation in $\Theta_e$ marked by $(x, (s_1,s_2))$ corresponds to the triangle in $\mathcal B$ whose vertices divide the circle into three arcs of lengths $(2\pi(1-x), 2\pi xs_1,2\pi xs_2)$. }
\end{figure}

The following observation plays a key role in this work. 
\begin{prop}\label{prop:bijection}
There is a bijection between the faces in $\mathcal B$ and the dislocations in $\Theta_e$. 
If a dislocation in $\Theta_e$ is marked by $(x ,(s_1,s_2))\in (0,1]\times \Delta$, then the corresponding triangle in $\mathcal B$ has edges of lengths $(2\sin(\pi x),2\sin(\pi xs_1),2\sin(\pi xs_2) )$. 
\end{prop}

A formal proof of Proposition \ref{prop:bijection} is given in Section \ref{section:Brownian}. This correspondence is illustrated in Figure \ref{fig_dual}. 
This bijection should be clear since the faces in $\mathcal B$ and the dislocations in $\Theta_e$ are both in bijection with the local minima of $e$. The second statement is simply obtained by basic geometry. By this bijection, if a triangle in $\B$ corresponds to a dislocation in $\Theta_e$ marked by $(x, (s_1,s_2))\in (0,1]\times \Delta$, then the length of its shortest edge is 
\begin{equation} \label{eq:psi'}
\psi'(x, (s_1,s_2)):=\min(2 \sin(\pi x), 2 \sin(\pi xs_1), 2 \sin(\pi xs_2)).
\end{equation}
Observing that the angle between the edge of length $2 \sin(\pi xs_1)$ and the edge of length $2 \sin(\pi xs_2)$ is $\pi(1-x)$, we find that the area of this triangle is
 \begin{equation} \label{eq:psi''}
 \psi''(x, (s_1,s_2)) :=2 \sin(\pi xs_1) \sin(\pi xs_2)\sin(\pi x).
\end{equation}
Hence with our fragmentation point of view, $N'(\epsilon)$ is the number of dislocations in $\Theta_e$ whose marks satisfy $\psi'(x, (s_1,s_2))>\epsilon$. A similar statement holds for $N''(\epsilon)$. In Section \ref{section:largedislocation}, we introduce the notion of {\it large dislocations} that generalizes these families of dislocations. We study the number of large dislocations in the context of a general self-similar fragmentation and obtain Theorem \ref{thm:large} below, which leads to the final proof of Theorem \ref{thm:Brownian}. We see that a phase transition appears in Theorem \ref{thm:large}, which explains the different limits in the two parts of Theorem \ref{thm:Brownian}. Our results on large dislocations are quite general which also enable us to answer the following two questions.



The first one is to study a generalization of the Brownian triangulation, the (geodesic) stable laminations of the disk introduced by Kortchemski \cite{kortchemski2014random}. For $\beta \in(1,2]$, the {\it $\beta$-stable lamination} is a random collection of non-crossing chords of the disk, which coincides with the Brownian triangulation when $\beta=2$, and is encoded by the normalized excursion of $\beta$-strictly stable L\'evy process when $\beta \in(1,2)$. For $\beta  \in(1,2)$, we find a bijection between the faces (which are not triangles) in the $\beta$-stable lamination and the dislocations in a certain self-similar fragmentation, which enables us to study the number of large faces in the $\beta$-stable lamination. 

 The second question is to determine the law of the length of the longest chord. For the Brownian triangulation, this has been calculated in \cite{aldous1994triangulating} by using discrete approximation by $T_n$; for the stable laminations, it is an open question due to Kortchemski, which is also mentioned in \cite{curien2014dissecting}. Noticing that the longest chord is an edge of the {\it centroid}, the (almost surely) unique face that contains the origin, we will answer this question by exploring the dislocation associated with the centroid.

 In short, we develop a study of the number of large dislocations in self-similar fragmentations and apply our results to estimate the number of faces in the Brownian triangulation and stable laminations. Our method also opens the way to study a number of other interesting problems. To mention just a few, we may consider the role of large dislocations in random recursive triangulations \cite{curienlegall2011random}, self-similar trees \cite{haasmiermont04genealogy} and quadtrees \cite{curienquadtrees}. 

The rest of the paper is organized as follows. 
In Section \ref{section:large}, we study the number of large dislocations in self-similar fragmentations. 
In Section \ref{section:Brownian}, we prove Theorem \ref{thm:Brownian} and find the law of the length of the longest chord in the Brownian triangulation. 
In Section \ref{section:stable}, we investigate the large faces and the longest chord in the stable laminations. 
In Section \ref{section:technical}, we complete the proofs of Lemma \ref{lem:m2} and Lemma \ref{lem:small}.


\section{Large dislocations in a self-similar fragmentation}\label{section:large}

In this section we study the number of large dislocations in self-similar fragmentations. The main result, Theorem \ref{thm:large}, is stated and proved in Section \ref{section:largedislocation}. Before that, we briefly review some basic facts about self-similar fragmentations in Section \ref{section:background}, and, in preparation for proving Theorem \ref{thm:large}, we explain how to change the index of self-similarity in Section \ref{section:index} and discuss the tagged fragment in Section \ref{section:tagged}.

\subsection{Background on self-similar fragmentations}\label{section:background}
We refer to \cite{berestycki2002ranked,bertoin2001homogeneous, bertoin2002self} for the general framework of self-similar fragmentations. Here we only give a short presentation. {\it A self-similar mass fragmentation with index of self-similarity $\alpha\in \R$} is a c\`adl\`ag Markov process $\X^{(\alpha)} = \left((X^{(\alpha)}_1(t), X^{(\alpha)}_2(t), \ldots ), ~t\geq 0\right)$ taking values in
\begin{equation}\Sd := \left\{\mathbf{s} = (s_1, s_2, s_3,  \cdots): 1\geq s_1 \geq s_2 \geq \cdots \geq 0,~and~ \sum_{i= 1}^{\infty} s_i \leq 1 \right\},\end{equation} 
which satisfies the {\it branching} and {\it scaling} properties. The branching property means that for every sequence $\mathbf{x} = (x_1,x_2, \cdots) \in \Sd$ and every $t\geq 0$, the distribution of $\X^{(\alpha)}$ given $\X(0)= \mathbf{x}$ is the same as the union of the masses, arranged in the decreasing order, of a sequence of independent fragmentations $(\X^{i})_{i\geq 1}$, where each $\X^{i}$ has distribution $\mathbb{P}_{x_i}$, the law of $\X^{(\alpha)}$ that starts from the state $x_i:=(x_i,0,\cdots)\in \Sd$. The scaling property means that for $x \in [0,1]$, the distribution of the re-scaled process $(x\X^{(\alpha)}(x^{\alpha}t))_{t\geq 0}$ under $\mathbb{P}_{1}$ is $\mathbb{P}_{x}$. 

For simplicity, throughout the rest of this paper we will implicitly suppose that any fragmentation starts from a single fragment with unit mass, and we will work under $\mathbb{P}:= \mathbb{P}_1$. 

A self-similar fragmentation is characterized by a triple $(\alpha, c, \nu)$: $\alpha\in \R$ is the index of self-similarity; the non-negative real constant $c$ is the {\it erosion rate}, which describes the speed at which the fragments melt continuously; the $\sigma$-finite measure $\nu$ on $\Sd$ verifying  
\begin{equation}
\nu(\{(1, 0,\cdots)\})=0, ~\text{and}~ \int_{\Sd}(1-s_1)\nu(d\s)<\infty
\end{equation} 
is the {\it dislocation measure}, which describes the statistics of the smaller pieces generated in a dislocation. For $ \s= (s_1,s_2, \cdots)\in \Sd$, a fragment of mass $x$ splits into masses $(xs_1, xs_2, \cdots)$ at rate $x^{\alpha}\nu(d\s)$. We say a fragmentation is {\it conservative} if its dislocation measure satisfies 
\begin{equation}
\nu \left( \s\in \Sd : \sum_{i=1}^{\infty}s_i < 1 \right)=0.
\end{equation}
Otherwise it is {\it dissipative}. 

A parallel notion is the {\it self-similar interval-partition fragmentations}, which was mentioned in the \hyperref[section:intro]{introduction}. An interval-partition fragmentation $\Theta = (\Theta(t), t\geq0) $ studies how the interval $(0,1)$ splits into smaller open intervals as time grows. If the existing intervals at $t>0$ form a sequence $(I_1(t), I_2(t),\cdots)$, arranged in the decreasing order of length, then the state of the interval-partition fragmentation at $t$ is their union $\Theta(t) = \bigcup_{i\in \N}I_i(t)$. Clearly $(\Theta_t, t\geq 0)$ is a family of nested open subsets of $(0,1)$. We observe that an interval-partition fragmentation naturally yields a mass fragmentation, specifically the length sequence process $(|I_1(t)|, |I_2(t)|,\cdots)_{t\geq 0}$. Therefore, we call $\Theta$ a self-similar interval-partition fragmentation if $\Theta$ is associated with a self-similar mass fragmentation.

\subsection{Changing the index of self-similarity}\label{section:index}
A self-similar fragmentation process with index of self-similarity zero is a {\it homogeneous fragmentation process}. 
For any self-similar fragmentation $\X^{(\alpha)}$ with no erosion and index of self-similarity $\alpha\in \R$, we are able to change the index $\alpha$ to $0$ by the following transformation introduced in \cite{bertoin2002self}. Let $\Theta^{(\alpha)}$ be an interval fragmentation whose associated mass fragmentation is $\X^{(\alpha)}$ as in Section \ref{section:background}. For $x\in(0,1)$ and $t\geq 0$, if $x\in \Theta^{(\alpha)}(t)$, then let $I^{(\alpha)}_x(t)$ be the interval component of $\Theta^{(\alpha)}(t)$ that contains $x$ at time $t$; if $x\not\in \Theta^{(\alpha)}(t)$, then by convention $I^{(\alpha)}_x(t):= \emptyset$. We define a family $T:= (T_x, x\in (0,1))$ by 
\begin{equation}\label{eq:timechange}
T_x(t):= \inf \left\{u\geq 0:  \int_0^u |I^{(\alpha)}_x(r)|^{\alpha}dr >t \right\}, \quad t\geq 0. 
\end{equation}
For $t\geq 0$, the set $\Theta^{(\alpha)}(T(t)):= \bigcup_{x\in(0,1)} I^{(\alpha)}_x(T_x(t))$ is open since it is the union of open intervals, and the family $(\Theta^{(\alpha)}(T(t)),t\geq 0)$ is nested. So we obtain a new interval-partition fragmentation $(\Theta(t))_{t\geq 0}:=(\Theta^{(\alpha)}(T(t))_{t\geq 0}$. According to Theorem 2 in \cite{bertoin2002self}, $\Theta$ is a homogeneous fragmentation with no erosion and the same dislocation measure $\nu$. Let $\X$ be the mass fragmentation associated with $\Theta$. We call $\X$ the homogeneous counterpart of $\X^{(\alpha)}$. 

In view of future use we state the following lemma, which is an extension of Equation (6) in \cite{bertoin2003asymptotic}.  

\begin{lem}\label{lem:index}
We consider a self-similar fragmentation $\X^{(\alpha)}= \left((X^{(\alpha)}_i(t))_{i\in \N}, t\geq 0\right)$ with no erosion and index of self-similarity $\alpha\in \R$, and its homogeneous counterpart $\X= \left((X_i(t))_{i\in \N}, t\geq 0\right)$. Let $f: [0,\infty) \to [0,\infty)$ be a measurable function, then the following equality holds: 
 \begin{equation}
\int_{0}^{\infty}\sum_{i=1}^{\infty} f(X_i(t))dt = \int_{0}^{\infty}\sum_{i=1}^{\infty} (X^{(\alpha)}_i(t))^{\alpha} f(X^{(\alpha)}_i(t))dt.
\end{equation}
\end{lem}
\begin{proof}
For every $t>0$, we have
\begin{equation}\sum_{i=1}^{\infty} f( X_i(t)) = \int_0^1|I^{(\alpha)}_x(T_x(t))|^{-1} f(|I^{(\alpha)}_x(T_x(t))|) dx.\end{equation}
Changing variable by $s = T_x(t)$, thus $dt =|I^{(\alpha)}_x(s)|^{\alpha} ds$, we have
\begin{equation}
\int_{0}^{\infty}\sum_{i=1}^{\infty} f( X_i(t)) dt 
= \int_0^1dx \int_{0}^{\infty}|I^{(\alpha)}_x(s)|^{\alpha-1} f(|I^{(\alpha)}_x(s)|) ds
= \int_{0}^{\infty} \sum_{i=1}^{\infty} (X^{(\alpha)}_i(s))^{\alpha} f( X^{(\alpha)}_i(s)) ds.
\end{equation}
\end{proof}
\begin{rem}\label{rem_Sigma}
We consider the homogeneous fragmentation $\X$ as above. For $p>1$, set 
\begin{equation}\label{eq:Sigma}
\Sigma(p):= \int_{0}^{+\infty} \sum_{i=1}^{\infty} X_i(r)^{p}dr.
\end{equation}
Lemma \ref{lem:index} implies that $\Sigma(p)$ has the same law as
\begin{equation}
\Sigma^{(1-p)}(1):= \int_{0}^{+\infty} \sum_{i=1}^{\infty} X^{(1-p)}_i(r)dr,
\end{equation}
where $\X^{(1-p)}$ is a self-similar fragmentation with index $1-p<0$, no erosion and the same dislocation measure $\nu$. The random variable $\Sigma^{(1-p)}(1)$ is called the {\it area of the fragmentation $\X^{(1-p)}$}, whose law is described by Theorem 2.1 in \cite{bertoin2012area}. Therefore, we also know the law of $\Sigma(p)$. In particular, we note that $\Sigma(p)$ has finite $k$-moment for $k\in \N$, see Lemma 3.1 in \cite{bertoin2012area}. 
\end{rem}

\subsection{The tagged fragment of a homogeneous fragmentation}\label{section:tagged}
 Let $\X= \left((X_i(t))_{i\in \N}, t\geq 0\right)$ be a homogeneous fragmentation with no erosion and dislocation measure $\nu$. Denote the natural filtration of $\X$ by $(\mathcal F_{t} = \sigma(X_i(s),s\leq t))_{t \geq 0}$.
In this section we recall some results about the tagged fragment taken from Section 4 of \cite{bertoin2002self}. 

As in Section \ref{section:background}, let $\Theta$ be an interval fragmentation whose associated mass fragmentation $\X$. In particular, $\Theta_t = \bigcup_{i\in \N} I_i(t)$, $t\geq 0$. Given a uniform random variable $V$ in $(0,1)$ which is independent of $\Theta$, the {\it tagged fragment} is the interval component that contains $V$. Denote the rank and the length of the tagged fragment at time $t$ respectively by $n(t)$ and 
\begin{equation}\chi(t):= |I_{n(t)}(t)| =\sum_{i=1}^{\infty} \ind{V\in I_i(t)}|I_i(t)| .\end{equation}
If $V\not\in \Theta(t)$, then let $n(t)=-1$ and $\chi(t)=0$ by convention. The tagged fragment is closely related to a subordinator.
\begin{lem}\label{lem:xi} 
The process $\xi=-\log \chi$ is a (possibly killed) $(\mathcal F_t)$-subordinator with Laplace exponent
\begin{equation}\label{eq:phi}
\Phi(p) := \int_{\Sd} \left(1- \sum_{i=1}^{\infty}s_i^{p+1} \right) \nu(d\s), \quad p\geq 0,
\end{equation} 
and its expected value is
\begin{equation}\label{eq:m}
\Exp{\xi(1)}= m:= \left\{
   \begin{array}{lr}
 \int_{\Sd} \sum_{i=1}^{\infty} s_i\log (s_i^{-1}) \nu(d\s) & \text{when $\nu$ is conservative}, \\
 +\infty  &\text{when $\nu$ is dissipative}. 
\end{array}\right.
\end{equation}
Let $dU$ be the potential measure of $\xi$, whose distribution function is
\begin{equation}U(x) := \Exp{\int_0^{\infty} \ind{\xi(t)\leq x} dt}, \quad x\geq 0, \end{equation}
then the Laplace transform of $dU$ is
\begin{equation}\label{eq:U}
\int_{0}^{\infty} \e^{-p x} dU(x) = \Phi(p)^{-1}, \quad p\geq 0.
\end{equation}
Further, let $f: [0,\infty) \to [0,\infty)$ be a measurable function with $f(0)=0$, then
\begin{equation}\label{eq:tagged}
\Exp{\int_{0}^{\infty}\sum_{i=1}^{\infty} f(X_i(t))dt} =\Exp{ \int_0^{\infty} \chi(t)^{-1} f(\chi(t))dt} = \int_0^{\infty} \e^{x} f(\e^{-x})dU(x) .
\end{equation}
\end{lem}
\begin{proof}
Theorem 3 in \cite{bertoin2001homogeneous} shows that $\xi$ is a subordinator with Laplace exponent \eqref{eq:phi}, therefore \eqref{eq:U} and \eqref{eq:m} follow as consequences. It is clear that conditionally on $\F_t$, the distribution of $\chi(t)$ is given by
\begin{equation}
\Probcond{\chi(t)=X_i(t)}{ \F_t} = X_i(t), \quad \forall i\in \N, \qquad \Probcond{\chi(t)=0}{\F_t} =  1 - \sum_{i=1}^{\infty} X_i(t),
\end{equation}
which yields \eqref{eq:tagged}. 
\end{proof}

We say $\xi$ is {\it lattice supported} if the law of $\xi(1)$ is a discrete measure supported by an arithmetic sequence including zero. It is clear that $\xi$ is not lattice supported if $\xi$ is not a compound Poisson process. 
\begin{defn}\label{defn:lattice}
The dislocation measure $\nu$ is ``{\it non-lattice}'' if $\xi$ is not lattice supported.  
\end{defn}

\subsection{Large dislocations in a self-similar fragmentation}\label{section:largedislocation}
We introduce ``large dislocations'' in this section. As in the \hyperref[section:intro]{introduction}, a dislocation of a fragmentation is labeled by $(x, \s = (s_1, s_2, \cdots))\in [0,1]\times \Sd$ if in this dislocation a fragment with size $x$ splits into a sequence of masses $(x\s)$. We recall from the definition of $\Sd$ that $(s_1, s_2, \cdots)$ in arranged in the decreasing order. 

\begin{defn}\label{defn:larged}
Let $\psi: [0,1] \times \Sd  \rightarrow [0, +\infty)$ be a measurable function with $\psi(0, \cdot) \equiv 0$. For $\epsilon>0$, a dislocation marked by $ (x, \s) \in (0,1]\times \Sd$ in a fragmentation process such that $\psi\left(x,\s \right)> \epsilon$ is called a {\it $(\psi,\epsilon)$-large dislocation}.
\end{defn}
This definition is motivated by the question about large triangles in Brownian triangulation in the \nameref{section:intro}. We note that by this definition, the number $N'(\epsilon)$ in Theorem \ref{thm:Brownian} is the number of $(\psi',\epsilon)$-large dislocations in the fragmentation $\Theta_e$ and $N''(\epsilon)$ is the number of $(\psi'',\epsilon)$-large dislocations in $\Theta_e$, where $\psi'$ and $\psi''$ are defined respectively by \eqref{eq:psi'} and \eqref{eq:psi''}, if we regard $\Delta$ as a subset of $\Sd$. 

We now state the main result of this section.

\begin{thm}\label{thm:large}
Consider a self-similar fragmentation $(\X^{(\alpha)}(t))_{t\geq 0} =(X^{(\alpha)}_1(t), X^{(\alpha)}_2(t), \cdots)_{t\geq 0}$ of index $\alpha$ with no erosion and dislocation measure $\nu$. Let $\psi$ be a function defined as in Definition \ref{defn:larged}, and denote by $\Nld(\epsilon)$ the total number of $(\psi,\epsilon)$-large dislocations in $\X^{(\alpha)}$. Suppose that $\psi$ can be expressed in the form
\begin{equation} \label{eq:psi}
\psi (x,\s) = \varphi(\s) x^b,
\end{equation}
where $\varphi: \Sd \to [0,\infty)$ is bounded and $b>0$. Define $\nubar\colon (0,\infty)\to [0,\infty]$ by 
\begin{equation}\label{eq:nubar}
\nubar (u): = \nu \left(\s \in \Sd: \varphi(\s )> u \right), \quad u> 0, 
\end{equation}
and consider respectively the following two (mutually exclusive) situations:  
\begin{enumerate}[label=(H\arabic{*}), ref=(H\arabic{*}),leftmargin=5.0em]
\item there exists $0<a<\frac{1}{b}$, such that $ \nubar (u)=o(u^{-a})$ as $u\to 0^+$, \label{item:H1}
\item there exists $a>\frac{1}{b}$ and $c>0$, such that $\nubar(u) \sim c u^{-a}$, as $u\to 0^+$.\label{item:H2}
\end{enumerate}
Note that if $\nu(\Sd)<\infty$, then \ref{item:H1} is always verified. 
\begin{enumerate}
\item If $\nu$ is non-lattice in the sense of Definition \ref{defn:lattice} and \ref{item:H1} holds, then
\begin{equation}
\lime \epsilon^{\frac{1}{b}} \Nld(\epsilon) = \frac{1}{m} \int_0^{\infty}  \nubar(u^b)du\cvL,
\end{equation}
where $m$ is defined as in \eqref{eq:m}, and by convention $\frac{1}{\infty}=0$.
\item If \ref{item:H2} holds, then
\begin{equation}
\lime \epsilon^{a}\Nld(\epsilon) = c \int_{0}^{+\infty} \sum_{i=1}^{\infty} X^{(\alpha)}_i(t)^{ab+\alpha}dt \cvL.
\end{equation}
\end{enumerate}
\end{thm}
\begin{rem} 
\begin{enumerate}
\item
Theorem \ref{lem:index} shows a phase transition when $b$ varies. If for $a>0$ and $c>0$,
\begin{equation}\nubar(u) = \nu(\s\in \Sd: \varphi(\s) >u) \sim c u^{-a}, \quad u\to 0^+,\end{equation} 
then the critical value of $b$ is $b_c = \frac{1}{a}$. In the sub-critical phase, the scaling limit is is a constant while in the super-critical phase the limit is a random variable. Notice that this phase transition is only possible when $\nu(\Sd)=\infty$. 
\item Theorem \ref{thm:large} does not cover the critical case, in which there exists $c>0$ such that $\nubar(u) \sim c u^{-\frac{1}{b}}$ as $u\to 0^+$. In the critical case, on the one hand by comparing with the sub-critical phase we see that $\lime \epsilon^{\frac{1}{b}+\delta} \Nld(\epsilon) = 0$ in $L^2(\mathbb P)$ for all $\delta> 0$, so $N(\epsilon)\in o(\epsilon^{-(\delta+ \frac{1}{b})})$; on the other hand, by comparing with the super-critical phase we have for all $\delta>0$ that in probability 
$$\liminf_{\epsilon \to 0} \epsilon^{\frac{1}{b}}\Nld(\epsilon) \geq c \int_{0}^{+\infty} \sum_{i=1}^{\infty} X^{(\alpha)}_i(t)^{\alpha+1- \delta}dt.$$
If the fragmentation is conservative, since it follows from Lemma \ref{lem:index} that 
$$\int_{0}^{+\infty} \sum_{i=1}^{\infty} X^{(\alpha)}_i(t)^{\alpha+1}dt =\int_{0}^{+\infty}1dt = +\infty,$$ then $\Nld(\epsilon) \not\in \bigO (\epsilon^{-\frac{1}{b}})$. However, we do not have a finer result. 
\item The functions $\psi'$ and $\psi''$ defined as in \eqref{eq:psi'} and \eqref{eq:psi''} do not have the form \eqref{eq:psi}, thus we cannot apply directly Theorem \ref{thm:large} to $N'(\epsilon)$ and $N''(\epsilon)$. We will show how to overcome this difficulty in Section \ref{section:Brownian}.
\end{enumerate}
\end{rem}

Before tackling the proof of Theorem \ref{thm:large}, let us look at an example of its application. It concerns {\it fragmentation trees} \cite{neininger2008fragtrees}. We consider a self-similar fragmentation $\X$ with non-lattice dislocation measure $\nu$ satisfying $\nu (\Sd) <\infty$, such that the fragmentation $\X$ has a discrete genealogical structure. Let us denote the {\it genealogical tree} by $\mathcal{U}:= \bigcup_{n=0}^{\infty} \mathbb{N}^n$, where $\mathbb{N}^0 = \{\emptyset\}$ by convention. Each $u\in \mathcal{U}$ is called an {\it individual}, we assign to each individual a fragment in the following way. The root $\emptyset$ represents the initial state and is marked by its mass $m_{\emptyset}=1$. Suppose that an individual $u \in \mathcal{U}$ stands for a fragment of mass $m_{u}>0$. Since $\nu (\Sd) <\infty$, this fragment lives for a strictly positive time before it splits. When it splits, it generates fragments of masses $(m_{(u, j)})_{j\in \N}$. Thus for $j\in \N$, the $j$-th child of $u$ is $(u, j)\in \mathcal{U}$ is the fragment of mass $m_{(u, j)}$, possibly zero. 
For $\epsilon>0$, let the {\it fragmentation tree} be the sub-tree of $\mathcal U$ consisting all nodes with mass greater $\epsilon$. Then the number of vertices in a fragmentation tree is the same as $\Nld(\epsilon)$, the number of $(\psi,\epsilon)$-large dislocations with $\psi: [0,1]\times\Sd \to [0,\infty)$ defined by $\psi(x, \s) = x$. Then $g(u)= \nu(\Sd)\ind{u<1}$ for $u> 0$. Thus \ref{item:H1} holds and it follows from Theorem \ref{thm:large} that 
\begin{equation} \lime \epsilon \Nld(\epsilon)= \frac{1}{m}\nu(\Sd)\cvL.\end{equation}
We refer to Theorem 1.3 in \cite{neininger2008fragtrees} for sharper results.

The rest of this section is devoted to the proof of Theorem \ref{thm:large}. We first point out that it suffices to consider homogeneous fragmentations. If Theorem \ref{thm:large} holds for the homogeneous fragmentations, noticing that the index changing transformation defined in Section \ref{section:index} preserves the number of $(\psi,\epsilon)$-large dislocations, then we can easily obtain the desired results for self-similar fragmentations with any index $\alpha\in \R$ by using this transformation and Lemma \ref{lem:index}. We left the details to the readers.

Hence we will focus on homogeneous fragmentations. We will show by Corollary \ref{cor_NA} below that it is equivalent to study $A(\epsilon)$ defined as in \eqref{eq:Ainf} below. Then we will study the asymptotic behavior of $A(\epsilon)$ as $\epsilon \to 0$ respectively when \ref{item:H1} holds or \ref{item:H2} holds, which finally proves Theorem \ref{thm:large}. 

\subsubsection{The compensated martingale}
Throughout the rest of Section \ref{section:large}, we consider a homogeneous fragmentation $\X = (X_i)_{i\in \N}$ with no erosion and dislocation measure $\nu$, and write $(\mathcal F_{t})_{t\geq 0}$ for its natural filtration.

A homogeneous fragmentation possesses a Poissonian structure which is described as follows. At every time $t>0$, there is at most one fragment that splits, we denote its index by $\kappa(t)$, and denote $\s(t)$ for the ratio between the masses of the ``children'' generated in this dislocation and their ``parent''. Then a dislocation is characterized by a triple $(t,\kappa(t), \s(t))\in [0,\infty)\times \N \times \Sd$. According to Theorem 9 in \cite{berestycki2002ranked}, the dislocations of $\X$ correspond to the atoms of a $(\mathcal F_{t})$-Poisson point process $(\kappa(t),\s(t))_{t\geq 0}$  in $\N \times \Sd$, with characteristic measure $\# \otimes \nu$, where $\#$ denotes the counting measure of $\N$. 
Using these notations, we express the number of $(\psi,\epsilon)$-large dislocations before time $t>0$ 
\begin{equation}
N_{t}(\epsilon) = \sum_{r\in(0,t] }\ind{\psi \left(X_{\kappa(r)}(r-),\s(r)\right)> \epsilon},
\end{equation} 
and the number of all $(\psi,\epsilon)$-large dislocations is $\Nld(\epsilon) = \lim_{t\to \infty} N_{t}(\epsilon)$. 

The Poissonian structure of the homogeneous fragmentation $\X$ permits us to introduce the compensated martingale. 
For $\epsilon>0$, define a function $f_{\epsilon} : [0,\infty) \to [0,\infty]$ by 
\begin{equation}\label{eq:f}
f_{\epsilon}(x):= \nu \left(\s \in \Sd: \psi(x,\s )> \epsilon \right), \quad x\in [0,\infty).
\end{equation}
Recall that $\psi(0, \cdot)\equiv 0$, thus $f_{\epsilon}(0) = 0$. Set 
 \begin{equation}\label{eq:Ainf}
\Ald (\epsilon) := \int_0^{\infty} \sum_{i=1}^{\infty} f_{\epsilon}(X_i(r)) dr. 
\end{equation}
If $\Exp{\Ald (\epsilon)}<\infty$, then it follows immediately from the compensation formula (see e.g. Section O.5 in \cite{bertoin1996Levy}) for the Poisson point process $(\kappa(t),\s(t))_{t\geq 0}$, that 
\begin{equation} M_t(\epsilon):= N_t(\epsilon)- \int_0^{t} \sum_{i=1}^{\infty} f_{\epsilon}(X_i(r)) dr, \quad t\geq 0,\end{equation}
is a uniformly integrable $(\F_t)_{t\geq 0}$-martingale. Further,
\begin{equation}M_t(\epsilon) \underset{t \to \infty}{\longrightarrow} \Nld(\epsilon)- \Ald(\epsilon) \quad \text{ a.s. and in } L^1(\mathbb{P}).\end{equation}
In particular, we have
\begin{lem}\label{lem:Mart}
If $\Exp{\Ald(\epsilon)}<\infty$, then $\Exp{ \Nld(\epsilon)}=\Exp{ \Ald(\epsilon)} <\infty$. 
\end{lem}
Further, by looking at the quadratic variation of the martingale $(M_t(\epsilon))_{t\geq 0}$, we have the following lemma. 
\begin{lem}\label{lem:qv}
If $\Exp{\Ald(\epsilon)}<\infty$, then 
\begin{equation}\Exp{ \left( \Nld(\epsilon) - \Ald(\epsilon) \right)^2} = \Exp{\Ald(\epsilon)}<\infty.\end{equation}
\end{lem}
\begin{proof}
Since $\epsilon$ is fixed, we do not indicate the dependence on $\epsilon$ for simplicity. 
If $\Exp{\Ald}<\infty$, then by Lemma \ref{lem:Mart}, $\Exp{\Nld} = \Exp{\Ald}<\infty$.
Noticing that $N_t$ and $\int_0^{t} \sum_{i=1}^{\infty} f_{\epsilon}(X_i(r)) dr$ are both increasing with respect to $t$, we have
\begin{equation}\Exp{\int_0^{\infty} |dM_t|} \leq \Exp{\Nld} + \Exp{\Ald} < \infty,\end{equation} 
i.e. the martingale $(M_t)_{t\geq 0}$ is of integrable variation. 

We first suppose that the martingale $(M_t)_{t\geq 0}$ is bounded in $\LP$. According to Lemma 36.2 in Chapter VI \cite{rogers2000diffusions}, since the martingale $(M_t)_{t\geq 0}$ is bounded in $L^2(\mathbb{P})$ and has integrable variation, its quadratic variation process is
\begin{equation}[M]_t = \sum_{r\in(0,t]} (M_r -M_{r-})^2 = \sum_{r\in(0,t] } \ind{\psi \left(X_{\kappa(r)}(r-),\s(r)\right)> \epsilon}^2 = N_t, \quad t\geq 0,\end{equation}
and $(M_t^2 - [M]_t)_{t\geq 0}$ is a uniformly integrable martingale. Thus 
\begin{equation}\Exp{(N-A)^2}= \lim_{t\to \infty} \Exp{M_{t}^2}  =  \lim_{t\to \infty} \Exp{[M]_{t}} = \Exp{N} =\Exp{A}, \end{equation} 
where the last equality follows from Lemma \ref{lem:Mart}. 


It thus remains to prove that $(M_t)_{t\geq 0}$ is indeed bounded in $\LP$. Let us consider a sequence of stopping times $(T_n)_{n\in \N}$ with $T_n:= \inf \{ t\geq 0: |M_t|>n\}$ (by convention $\inf \emptyset = +\infty$). 
 For every fixed $n\in \N$, the martingale $(M_{T_n \wedge t})_{t\geq 0}$ is bounded, thus it follows from the arguments above that $\Exp{\ M^2_{T_n\wedge t}} = \Exp{ N_{T_n \wedge t} }$ for every $t\geq 0$. Then we have by Fatou's lemma that for every $t\geq 0$
$$ \Exp{M^2_t} \leq \liminf_{n\to \infty} \Exp{M^2_{T_n\wedge t}} =\liminf_{n\to \infty}\Exp{ N_{T_n \wedge t} } \leq \Exp{N} =\Exp{A}.$$
So we have $\sup_{t\geq 0}\Exp{M^2_t}\leq \Exp{A}<\infty$, which completes the proof. 
\end{proof}

\begin{cor}\label{cor_NA}
For $\lambda>0$, if $\epsilon^{\lambda}\Ald(\epsilon)$ converges in $\LP$ as $\epsilon \to 0$, then $\epsilon^{\lambda}\Nld(\epsilon)$ converges in $\LP$ to the same limit as $\epsilon \to 0$.
\end{cor}
\begin{proof}
Set 
\begin{equation}A_0: = \lime \epsilon^{\lambda}\Ald(\epsilon)\cvL.\end{equation}
By Lemma \ref{lem:qv}, we have
\begin{align}
\Exp{(\epsilon^{\lambda}\Nld(\epsilon) - A_0)^2} &\leq 2 \Exp{(\epsilon^{\lambda}\Nld(\epsilon) -\epsilon^{\lambda} \Ald(\epsilon))^2} + 2 \Exp{(\epsilon^{\lambda}\Ald(\epsilon) - A_0)^2} \\
&= 2\epsilon^{2\lambda}\Exp{\Ald(\epsilon)} + 2 \Exp{(\epsilon^{\lambda}\Ald(\epsilon) - A_0)^2}.
\end{align}
Then the claim holds since 
\begin{equation}\lim_{\epsilon \to 0}2\epsilon^{2\lambda}\Exp{\Ald(\epsilon)} + 2 \Exp{(\epsilon^{\lambda}\Ald(\epsilon) - A_0)^2}=0.\end{equation}
\end{proof}
By Corollary \ref{cor_NA}, to study the asymptotic behavior of $\Nld(\epsilon)$ as $\epsilon \to 0$, it suffices to study the asymptotic behavior of $\Ald(\epsilon)$ as $\epsilon \to 0$. 

\subsubsection{The case when \ref{item:H1} holds}\label{section:Alim}
Now we study the asymptotic behavior of $\Ald(\epsilon)$ as $\epsilon \to 0$. Suppose that $\psi$ has the form \eqref{eq:psi}, then by the definitions of $f_{\epsilon}$ and $\nubar$ in \eqref{eq:f} and \eqref{eq:nubar}, we can rewrite $\Ald(\epsilon)$ by 
\begin{equation}\label{eq:psi1}
\Ald(\epsilon)  =\int_0^{\infty} \sum_{i=1}^{\infty} f_1(\epsilon^{-\frac{1}{b}} X_i(t)) dt, 
\end{equation}
and we have
\begin{equation}\label{eq:psi2}
f_1(x) = g (x^{-b}), \quad x>0.
\end{equation}
We first suppose that \ref{item:H1} holds. Let us explain briefly the motivation of considering \ref{item:H1}. By \eqref{eq:tagged}, we have
\begin{equation}\label{eq:lem:m1}
  \Exp{\Ald(\epsilon)} =\epsilon^{-\frac{1}{b}} \int_0^{\infty} f_1 \left(\e^{(-\frac{1}{b}\log \epsilon-x)}\right)\e^{ - (-\frac{1}{b}\log \epsilon -x)} dU(x).
\end{equation}
Recall that $dU$ is the potential measure of the subordinator $\xi$ as in Lemma \ref{lem:xi}, we can study the limit of the right hand side when $\epsilon \to 0$ with the help of the renewal theorem for subordinators (see for example Proposition 1.6 in \cite{bertoin1999subordinators}). To use the renewal theorem, we need the function defined by $f_1(\e^{x})\e^{-x} = g(e^{-bx})e^x$ to be directly Riemann integrable on $\R$. Hence it is natural to consider condition \ref{item:H1}, which ensures this integrability. In order to use the renewal theorem, we also suppose that $\nu$ is non-lattice in the sense of Definition \ref{defn:lattice}.

\begin{lem}\label{lem:m1}
Suppose that $\nu$ is non-lattice. If $\psi$ has the form \eqref{eq:psi} and \ref{item:H1} holds, then 
\begin{equation}
 \lim_{\epsilon\to 0} \epsilon^{\frac{1}{b}} \Exp{\Ald(\epsilon)} = \frac{1}{m} \int_0^{\infty} g(u^b)du,
\end{equation}
 where $m$ is defined as in \eqref{eq:m}, and by convention $\frac{1}{\infty}=0$.
\end{lem}
\begin{proof} 
If $\psi$ has the form \eqref{eq:psi} and \ref{item:H1} holds, then there exists $\bar{c}>0$ such that 
\begin{equation}\label{eq:cbarf}
f_1(x) = g(x^{-b}) \leq \bar{c} \ind{x\geq |\varphi|^{-\frac{1}{b}} } x^{ab}, \quad x>0, 
\end{equation}
where $|\varphi|$ stands for the $L^{\infty}$ norm of $\varphi$.
It follows that $\tilde{f} \colon [0,\infty) \to [0,\infty)$ defined by
\begin{equation}\tilde{f}(y):=f_1(\e^{y}|\varphi|^{-\frac{1}{b}} )\e^{-y}|\varphi|^{\frac{1}{b}} \leq \bar{c}|\varphi|^{\frac{1}{b} - a}  \e^{(ab-1)y},\quad y\in [0,\infty),\end{equation} 
is directly Riemann integrable on $[0,\infty)$. Observing that $f_1(x)=0$ for all $x< |\varphi|^{-\frac{1}{b}}$, we write \eqref{eq:lem:m1} in terms of $\tilde{f}$ and obtain that
$$\Exp{A(\epsilon)} =\epsilon^{-\frac{1}{b}} \int_0^{\frac{1}{b}\log |\varphi| -\frac{1}{b}\log \epsilon} \tilde{f} \left(\frac{1}{b}\log |\varphi| -\frac{1}{b}\log \epsilon-y  \right) dU(y). $$
As $-\frac{1}{b}\log\epsilon \to +\infty$ when $\epsilon \to 0$, by the renewal theorem (Proposition 1.6 in \cite{bertoin1999subordinators}) \footnote{More precisely, we use the integral version of the renewal theorem, also known as the key renewal theorem, which can be derived from Proposition 1.6 in \cite{bertoin1999subordinators} by using the same arguments as the proof of Theorem 4.4.5 in \cite{Durrett}.} for subordinator $\xi$, we have
\begin{linenomath}\begin{align}
   \lim_{\epsilon \to 0}\epsilon^{\frac{1}{b}}\Exp{A(\epsilon)} & = \frac{1}{\Exp{\xi(1)}} \int_{0}^{+\infty} \tilde{f} (y)dy  = \frac{1}{\Exp{\xi(1)}} \int_{-\infty}^{+\infty} \e^{-x} f_1 ( \e^{x})dx = \frac{1}{m}\int_0^{+\infty}g(u^b)du,
\end{align}\end{linenomath}
where we have changed variables $x= y -\frac{1}{b} \log |\varphi|$ and $u = \e^{-x}$ to get the second equality and the third equality respectively. 
 \end{proof}

Further we can prove the following result.  
\begin{lem}\label{lem:m2}
Suppose that $\nu$ is non-lattice. If $\psi$ has the form \eqref{eq:psi} and \ref{item:H1} holds, then we have
\begin{equation}\lim_{\epsilon \to 0}\epsilon^{\frac{2}{b}} \Exp{\Ald(\epsilon)^2}= \left( \frac{1}{m} \int_0^{\infty} g(u^b)du \right)^2.\end{equation}
\end{lem} 
We postpone the proof of Lemma \ref{lem:m2} to Section \ref{section:technical}. Our arguments proceed in a similar way as in the proof of Lemma 5 in \cite{bertoin2005energy}.

Now we are able to prove Theorem \ref{thm:large} for the case when \ref{item:H1} holds. 
\begin{proof}[Proof of Theorem \ref{thm:large}: when \ref{item:H1} holds.]
By Lemma \ref{lem:m1} and Lemma \ref{lem:m2}, we have
\begin{equation}\lim_{\epsilon\to 0} \Exp{ \left( \epsilon^{\frac{1}{b}}\Ald(\epsilon)  -  \frac{1}{m} \int_0^{\infty} g(u^b)du \right)^2} =0 \end{equation}
then the claim follows from Corollary \ref{cor_NA}. 
\end{proof}
In the same spirit as Lemma \ref{lem:m2}, we introduce the following lemma for future use in Section \ref{section:Brownian}. We postpone its proof to Section \ref{section:technical}.
\begin{lem}\label{lem:small}
Consider the $(\psi,\epsilon)$-large dislocations of fragments of masses greater than $\frac{1}{2}$, and denote their number by $\bar{N}(\epsilon)$. If $\psi$ has the form \eqref{eq:psi} and \ref{item:H1} holds, then 
\begin{align}
\lime \epsilon^{\frac{1}{b}} \bar{N}(\epsilon) = 0\cvL.
\end{align}
\end{lem} 
\subsubsection{The case when \ref{item:H2} holds}
We still suppose that $\psi$ has the form \eqref{eq:psi}, which implies that \eqref{eq:psi1} and \eqref{eq:psi2} hold. Now we turn to the situation when \ref{item:H2} holds. This situation differs significantly from the case when \ref{item:H1} holds: now the function $f_1(e^x)e^{-x}$ is not directly Riemann integrable on $\R$, thus we cannot obtain the result in Lemma \ref{lem:m1}. However, \ref{item:H2} implies that $f_1(\epsilon^{-\frac{1}{b}}x) = \nubar(\epsilon x^{-b}) \sim c \epsilon^{a} x^{ab}$ as $\epsilon \to 0$. As $ab>1$, recall that $\Sigma(ab)$ defined as in \eqref{eq:Sigma} is square integrable. Thus intuitively $\Ald(\epsilon)\sim c \epsilon^{-a} \Sigma(ab)$ as $\epsilon \to 0$. 
To give a rigorous proof, let us introduce the following lemma. 

\begin{lem}\label{lem:general}
Recall $f_{\epsilon}$ from \eqref{eq:f}. Suppose that there exist $\rho >1$, $\bar{c}>0$ and $\lambda >0$, such that for every $x\in [0,1]$, 
\begin{equation}\label{eq:general1}
\epsilon^{\lambda} f_{\epsilon}(x) \leq \bar{c} x^{\rho}, \quad \text{for all } \epsilon>0, 
\end{equation}
and that for every $x\in [0,1]$, $f_{*}(x):=\lim_{\epsilon\to 0} \epsilon^{\lambda}f_{\epsilon}(x)$ exists.
Then we have
\begin{equation}
\lime \epsilon^{\lambda}\Ald(\epsilon) = \int_0^{+\infty} \sum_{i=1}^{\infty} f_{*}\left( X_i(t)\right) dt \cvL.
\end{equation}
\end{lem} 
 We note that this lemma does not require $\psi$ to have the form \eqref{eq:psi}. We also remark that although $\lambda$ is not unique, the only interesting choice of $\lambda$ is the one such that $f_{*} \not\equiv 0$.
\begin{proof}
For $t\geq 0$, it follows from \eqref{eq:general1} that
\begin{equation}
\epsilon^{\lambda}\sum_{i=1}^{\infty} f_{\epsilon}(X_i(t))  \leq \bar{c}\sum_{i=1}^{\infty} X^{\rho}_i(t) .
\end{equation}
Recall that $\Sigma(\rho)$ defined as in \eqref{eq:Sigma} is a square integrable random variable, thus $\mathbb P$-almost surely $\bar{c} \Sigma(\rho)$ is finite. Integrate the left-hand side with respect to $t$, then let $\epsilon \to 0$. Hence using the dominated convergence theorem, we get that $\mathbb P$-almost surely
\begin{equation}\lim_{\epsilon \to 0}\epsilon^{\lambda}\Ald(\epsilon) = \lim_{\epsilon \to 0}\epsilon^{\lambda}\int_0^{\infty} \sum_{i=1}^{\infty} f_{\epsilon}(X_i(t))dt = \int_0^{\infty}\sum_{i=1}^{\infty}f_{*}(  X_i(t))dt \leq \bar{c} \Sigma(\rho).\end{equation} 
Using the dominated convergence theorem, we have
\begin{equation}
 \lim_{\epsilon \to 0} \Exp{\left( \epsilon^{\lambda}\Ald(\epsilon) - \int_{0}^{\infty}\sum_{i=1}^{\infty} f_{*}(  X_i(t)) dt\right)^2} =0.
\end{equation} 

\end{proof}

\begin{cor}\label{cor_H2}
 If $\psi$ has the form \eqref{eq:psi} and \ref{item:H2} holds, then 
\begin{equation}
 \lime \epsilon^{a}\Ald(\epsilon) = c \Sigma(ab)\cvL.
\end{equation}
where $\Sigma(ab)$ is defined as in \eqref{eq:Sigma}.
\end{cor}
\begin{proof}
If \ref{item:H2} holds, then there exists $\bar{c}>0$ such that 
\begin{equation}g(u) = g(u)\ind{ u \leq |\varphi|} \leq \bar{c} \ind{ u \leq |\varphi|}  u^{-a}, \quad \text{for all } u>0.\end{equation}
where $|\varphi|$ stands for the $L^{\infty}$ norm of $\varphi$. Thus, recalling $f_{\epsilon}$ from \eqref{eq:f}, for every $x\in [0,1]$ we have
\begin{equation}\epsilon^a  f_{\epsilon} (x)= \epsilon^a  \nubar (\epsilon x^{-b})\leq \bar{c} x^{ab}, \quad \text{for all } \epsilon>0.\end{equation}
Further, \ref{item:H2} yields that for every $x\in [0,1]$,
\begin{equation}\lim_{\epsilon \to 0} \epsilon^a  f_{\epsilon} (x) = \lim_{\epsilon\to 0}\epsilon^a  g (\epsilon x^{-b}) = c x^{ab}.\end{equation}
Hence the claim follows from Lemma \ref{lem:general}.
\end{proof}

Now we complete the proof of Theorem \ref{thm:large}. 
\begin{proof}[Proof of Theorem \ref{thm:large}: when \ref{item:H2} holds.]
Recall that we have reduced the proof to the homogeneous case, $\alpha=0$. Then the result follows from Corollary \ref{cor_H2} and Corollary \ref{cor_NA}. 
\end{proof}

\section{The Brownian triangulation}\label{section:Brownian}
In this section, we come back to the Brownian triangulation $\B$. We will prove Theorem \ref{thm:Brownian} and find the law of the length of the longest chord in $\B$.

Recall from the \hyperref[section:intro]{introduction} that $\B$ is encoded by a normalized Brownian excursion $e$, i.e. almost surely 
\begin{equation}\mathcal B= \bigcup_{s \overset{e}{\sim} t, s,t\in[0,1)}[{\rm e}^{i2\pi s}, {\rm e}^{i2\pi t}], \end{equation}
 where $s \overset{e}{\sim} t$ if and only if $e(s)=e(t)=\min_{r\in [s\wedge t, s\vee t]}e(r)$.
We have also introduced a fragmentation process 
 $$\Theta_e(t)= \{s\in (0,1): e_s>t\}, \quad t\geq 0.$$
Let us first give a formal proof of Proposition \ref{prop:bijection}, that there exists a bijection between the faces of $\B$ and the dislocations in $\Theta_e$.
\begin{proof}[Proof of Proposition \ref{prop:bijection}]
For every $s\in (0,1)$, we write $cl_{e}(s)$ for the equivalence class under relation $\overset{e}{\sim}$. 
We observe that for each $s\in (0,1)$, $cl_{e}(s)$ contains at most three points, since $\mathcal B$ is a triangulation. 

Now let us prove the bijection between the triangles of $\B$ and the dislocations in $\Theta_e$.
Suppose that a triangle of $\mathcal B$ has vertices $\e^{i 2\pi s_1}$, $\e^{i 2\pi s_2}$ and $\e^{i 2\pi s_3}$ with $s_1< s_2 <s_3$, then $s_1\overset{e}{\sim} s_2 \overset{e}{\sim} s_3$ and thus $cl_{e}(s_1) = \{ s_1, s_2 ,s_3\}$, because $cl_{e}(s_1)$ cannot contain more than three points. On the one hand, by the definition of $\overset{e}{\sim}$ we see that $e(s_1)= e(s_2)= e(s_3)$. On the other hand, for every $r \in (s_1, s_2)\cup (s_2, s_3)$ we must have $e(r)>e(s_1)$: because otherwise there is $e(r) = e(s_1)$ and thus $r \in cl_{e}(s_1)$, which is impossible. Therefore, at time $e(s_1)$ there is a dislocation of $\Theta_e$, the interval $(s_1, s_3)$ splits into $(s_1, s_2)$ and $(s_2, s_3)$. 

Conversely, if in $\Theta_e$ there is a dislocation that an interval $(s_1, s_3)$ splits into $(s_1, s_2)$ and $(s_2, s_3)$, then $s_1\overset{e}{\sim} s_2 \overset{e}{\sim} s_3$ and thus $cl_{e}(s_1) = \{ s_1, s_2 ,s_3\}$. So there is triangle of $\mathcal B$ with vertices $\e^{i 2\pi s_1}$, $\e^{i 2\pi s_2}$ and $\e^{i 2\pi s_3}$.
\end{proof}

According to \cite{bertoin2002self}, the fragmentation process $\Theta_e$ has index of self-similarity $-\frac{1}{2}$ and no erosion. Its dislocation measure $\nu_e$, binary and conservative, is specified by 
\begin{equation}
\nu_e(ds_1)= \frac{2}{\sqrt{2\pi s_1^3(1-s_1)^3}}ds_1 ~~1/2\leq s_1<1.
\end{equation}
Let us regard $\nu_e$ as a measure on $\Sd$, thus $\nu_e$ is supported on
\begin{equation}\left\{ \s= (s_1, s_2 , 0, \cdots) \in \Sd~:~  s_1 + s_2=1 \right\}\subset \Sd .\end{equation}
It is clear that $\nu_e$ is non-lattice in the sense of Definition \ref{defn:lattice}. By calculation we also find the quantities defined as in Lemma \ref{lem:xi}:
\begin{equation}\label{eq:mBT}
m =\int_{\Sd} \sum_{i=1}^{\infty} s_i\log(s_i^{-1}) \nu_e(d\s)= 2\sqrt{2\pi},
\end{equation}
\begin{equation}\label{eq:phiBT}
  \Phi(p)= \int_{\Sd} \left(1- \sum_{i=1}^{\infty}s_i^{p+1}\right) \nu_e(d\s) = 2\sqrt{2} \frac{\Gamma(p + 1/2)}{\Gamma(p)}, \quad p>0,
\end{equation}
where $\Gamma$ is the Gamma function.

\subsection{Proof of Theorem \ref{thm:Brownian}}
In the \hyperref[section:intro]{introduction}, we have marked a dislocation in $\Theta_e$ by $(x, (s_1,s_2))\in (0,1]\times \Delta$, if in this dislocation an interval of length $x$ splits into two pieces of respective lengths $xs_1$ and $xs_2$. 
To make notations consistent with Section \ref{section:large}, let us mark a dislocation in $\Theta_e$ by $(x, (s_1,s_2,0, \cdots))\in (0,1]\times \Sd$ from now on. 
 \begin{proof}[Proof of Theorem \ref{thm:Brownian}]
 {\bf 1.}
In our fragmentation point of view, $N'(\epsilon)$ equals the number of all $(\psi', \epsilon)$-large dislocations of $\Theta_e$, where $\psi'$ as in \eqref{eq:psi'} is defined by
\begin{equation}\psi'(x,\s)=\min(2 \sin(\pi x), 2 \sin(\pi xs_1), 2 \sin(\pi xs_2)),\quad \markd.\end{equation}
We cannot directly apply Theorem \ref{thm:large} to $\Nld'(\epsilon)$, because $\psi'$ is not of form \eqref{eq:psi}. Let us consider $\psi_1: [0,1] \times \Sd \to \R_+$, a function defined by $\psi_1(x, \s)= (1-s_1) x $, $\markd$. This function is of form \eqref{eq:psi}. Hence we will study $N^{\psi_1}(\epsilon)$, the number of $(\psi_1,\epsilon)$-large dislocations of $\Theta_e$, and compare $\Nld'(\epsilon)$ with $N^{\psi_1}(\epsilon)$.

Recall that $s_1\geq s_2$ by the definition of space $\Sd$. On the one hand, if $x\leq \frac{1}{2}$, then $\sin(\pi xs_2) \leq \sin(\pi xs_1) \leq \sin(\pi x)$, thus 
\begin{equation}\left \{(x ,\s) \in (0,1]\times \Sd:  2 \sin(\pi xs_2) > \epsilon, x\leq \frac{1}{2} \right\} \subset\left \{(x ,\s) \in (0,1]\times \Sd: \psi'(x,\s) > \epsilon \right\}.\end{equation}
On the other hand, it is plain that $\psi'(x,\s) \leq 2 \sin(\pi xs_2)$ and thus
\begin{equation}\left\{(x ,\s) \in (0,1]\times \Sd: \psi'(x,\s)> \epsilon  \right\} \subset \left\{(x ,\s) \in (0,1]\times \Sd: 2 \sin(\pi xs_2)> \epsilon  \right\}. \end{equation}
We may assume that $\epsilon < 2$, so $\arcsin (\epsilon/2)$ is well-defined. Let $\bar{N}^{\psi_1}(\epsilon)$ be the number of the dislocations which are $(\psi_1,\epsilon)$-large and whose marks $(x,\s)\in (0,1]\times \Sd$ satisfy $x >\frac{1}{2}$. Then the above observations yield
\begin{equation}\label{eq:N'}
N^{\psi_1}(\pi^{-1}\arcsin (\epsilon/2))- \bar{N}^{\psi_1}(\pi^{-1}\arcsin (\epsilon/2)) \leq N'(\epsilon) \leq N^{\psi_1}(\pi^{-1}\arcsin (\epsilon/2)).
\end{equation}

Let us look at $N^{\psi_1}(\epsilon)$. Since $\psi_1$ is of form \eqref{eq:psi} and $g_1$ defined as in \eqref{eq:nubar} is
\begin{equation}\label{eq:gBT}
g_1 (u):= \nu_e(\s \in \Sd: 1- s_1>u) = \frac{2\sqrt{2}}{\sqrt{\pi}} u^{-\frac{1}{2}} \frac{1-2u}{ \sqrt{1-u }}\ind{u<\frac{1}{2}}  \sim \frac{2\sqrt{2}}{\sqrt{\pi}} u^{-\frac{1}{2}}, \quad u\to 0^+,
\end{equation}
the hypothesis \ref{item:H1} is satisfied. We also find that $\int_0^{\infty} g_1(u) du =\left. \frac{4\sqrt{2}}{\sqrt{\pi}} \sqrt{(1-u)u}\right|_0^{\frac{1}{2}} = \frac{2\sqrt{2}}{\sqrt{\pi}}$. Thus using Theorem \ref{thm:large} and \eqref{eq:mBT} yields
\begin{equation}
\lime \epsilon N^{\psi_1}(\epsilon)= \frac{1}{\pi}\cvL.
\end{equation}
Further, since \ref{item:H1} holds, applying Lemma \ref{lem:small} to $\bar{N}_{\infty}^{\psi_1}(\epsilon)$, we have
\begin{equation}
\lime \epsilon \bar{N}^{\psi_1}(\epsilon)= 0 \cvL.
\end{equation}
Combining these two limits and \eqref{eq:N'}, we prove the claim. 

{\bf 2.} 
Recall that $N''(\epsilon)$ is the number of all $(\psi'', \epsilon)$-large dislocations of $\Theta_e$, where $\psi''$ defined as in \eqref{eq:psi''} is
\begin{equation} \psi''(x, \s) :=2 \sin(\pi xs_1) \sin(\pi xs_2)\sin(\pi x), \quad \markd.\end{equation}
This function is not of form \eqref{eq:psi}, thus we cannot use Theorem \ref{thm:large} directly. However, we will study $N''(\epsilon)$ by using similar arguments as in the proof of Theorem \ref{thm:large}.

Because the transformation explained in Section \ref{section:index} does not affect the number of total large dislocations, we regard $\Nld''(\epsilon)$ as the number of $(\psi'',\epsilon)$-large dislocations in $\X$, the homogeneous mass fragmentation counterpart of $\Theta_e$. 
For $(x, (s_1,s_2, 0,\cdots))\in (0,1]\times \Sd$ such that $s_1+s_2=1$, by the trigonometric addition formula $\cos(z_1 - z_2)- \cos(z_1+z_2) = 2\sin(z_1)\sin(z_2)$, we have
\begin{linenomath}\begin{align}
\psi''(x, \s) = 2 \sin(\pi xs_1) \sin(\pi xs_2)\sin(\pi x) = \left(\cos(\pi x -2\pi x s_2) - \cos(\pi x) \right) \sin(\pi x).
\end{align}\end{linenomath}
Hence for every $0<x\leq 1$, $\psi''(x,\s)>\epsilon$ if and only if
\begin{equation}
1- s_1 = s_2 >  h''(x,\epsilon) :=\frac{1}{2} -\frac{1}{2\pi x}\arccos\left(\min \left(\cos \pi x  +\frac{\epsilon}{\sin \pi x}, 1\right)\right).
\end{equation}
Note that when $\epsilon \geq \sin(\pi x) (1 - \cos(\pi x))$, this inequality reads $s_2>\frac{1}{2}$ when means for all $\s\in \Sd$ it is impossible to have $\psi''(x,\s)>\epsilon$.
We want to use Lemma \ref{lem:general} to $A''(\epsilon)$ defined as in \eqref{eq:Ainf}:
\begin{equation}
A''(\epsilon) = \int_0^{\infty} \sum_{i=1}^{\infty} \nu_e \left( \s\in \Sd: \psi''(X_i(t), \s) >\epsilon \right) dt.
 \end{equation}
So we check the two assumptions in Lemma \ref{lem:general}. On the one hand, it is clear that for every $x>0$, 
\begin{equation}
\lim_{\epsilon \to 0} h''(x,\epsilon)\epsilon^{-1} =(2\pi x \sin^2 \pi x)^{-1},
\end{equation}
then using the function $g_1$ defined in \eqref{eq:gBT}, we have for every $\epsilon>0$
\begin{equation}
\lim_{\epsilon \to 0} \epsilon^{\frac{1}{2}}\nu_e \left( \s\in \Sd: \psi''(x, \s) >\epsilon \right) = \lim_{\epsilon \to 0} \epsilon^{\frac{1}{2}} g_1 (h''(x,\epsilon) ) = 4 x^{\frac{1}{2}} \sin \pi x.
\end{equation}
On the other hand, observing that 
\begin{equation}
 \psi''(x, \s) = 2 \sin(\pi xs_1) \sin(\pi xs_2)\sin(\pi x)\leq 2 \pi^3 x^3 s_2, 
\end{equation}
and $g_1 (u) \leq \frac{2\sqrt{2}}{\sqrt{\pi}} u^{-\frac{1}{2}}$ for every $u>0$, we have
\begin{equation} 
\epsilon^{\frac{1}{2}}\nu_e \left( \s\in \Sd: \psi''(x, \s) >\epsilon \right) \leq \epsilon^{\frac{1}{2}}g_1(\epsilon (2 \pi^3 x^3)^{-1}  ) \leq 4 \pi x^{\frac{3}{2}}.
\end{equation}
Hence it follows from Lemma \ref{lem:general} and Corollary \ref{cor_NA} that
 \begin{equation}
\lime \epsilon^{\frac{1}{2}}N''(\epsilon) = \int_{0}^{\infty}\sum_{i=1}^{\infty} 4 X_i(t)^{\frac{1}{2}}\sin( \pi X_i(t)) dt \cvL.
\end{equation}
By applying Lemma \ref{lem:index} to $\Theta_{e}$ (with index of self-similarity $-\frac{1}{2}$) and its homogeneous counterpart $\X$, we rewrite the right-hand side in terms of $\Theta_{e}(t) = \bigcup_{i\in \N} I_i(t)$ and obtain the desired result.
\end{proof}

\subsection{The length of the longest chord}

In \cite{aldous1994triangulating}, Aldous has determined the law of the length of the longest chord by approximating the Brownian triangulation $\B$ by discrete uniform triangulations of polygons, studying uniform triangulations and then passing to the limit. Here we propose another approach using the bijection in Proposition \ref{prop:bijection}.  

We will make use of the {\it centroid}, the face that contains the origin, since it is plain that the longest chord must be an edge of the centroid.
Almost surely, no chord in $\B$ passes through the origin thus the centroid in $\B$ is unique. Let us consider the dislocation in $\Theta_e$ associated with the centroid. By Proposition \ref{prop:bijection}, if it is marked by $(x,\s)\in (0,1]\times \Sd$, then the vertices of the centroid divide the unit circle into three arcs whose lengths are $(2\pi(1-x),~ 2\pi x s_1,~  2\pi x s_2)$. Due to the property of the centroid, each of these arcs has length less than $\pi$, then
\begin{equation}\label{eq:centroid}
(1-x)<\frac{1}{2} ~\&~ x s_1<\frac{1}{2}~\&~    x s_2 <\frac{1}{2} \quad \Longleftrightarrow \quad x>\frac{1}{2}~\&~ xs_1 < \frac{1}{2}\quad \Longleftrightarrow \quad \min(x, 1-xs_1)> \frac{1}{2} .
\end{equation}
Conversely, it is easy to see that if a dislocation in $\Theta_e$ is marked by $(x,\s)\in (0,1]\times \Sd$ verifying the above relation, then it must correspond to the centroid. By further study of the centroid we have the following observation. 
\begin{lem}\label{lem:longest}
Let $2\pi L$ be the length of the minor arc with the same endpoints as the longest chord in $\B$, then $L\leq \frac{1}{2}$ and the longest chord has length $2\sin(\pi L)$. Define a function $\psi_L: [0,1] \times \Sd \to [0,\infty)$ by $\psi_L(x, \s) = \min(x, 1-xs_1)$. For $a\leq \frac{1}{2}$, let $\Nld(1- a)$ be the number of $(\psi_L, 1-a)$-large dislocations in $\Theta_e$ as in Definition \ref{defn:larged}, then 
\begin{equation}
\Prob{L < a} = \Exp{\Nld(1- a)}
\end{equation}
\end{lem}

\begin{proof}
A dislocation is $(\psi_L, 1-a)$-large if and only if its mark $(x,\s)\in (0,1]\times \Sd$ satisfies
 \begin{equation}\label{eq:La}
 \min(x, 1-xs_1)> 1-a \quad \Longleftrightarrow \quad  x>1-a~\&~ xs_1 <a \quad \Longleftrightarrow \quad (1-x)<a ~\&~ x s_1<a~\&~    x s_2 <a. 
\end{equation}
As $ (1-a) \geq \frac{1}{2}$, if there is a $(\psi_L, 1-a)$-large dislocation in $\Theta_e$, then it is associated with the centroid by \eqref{eq:centroid}. In particular, almost surely $\Nld(1- a) =1$ or $0$. 

Let us consider the the dislocation in $\Theta_e$ associated with the centroid in $\B$, whose mark is $\markd$. The longest chord in $\B$ must be an edge of the centroid, thus $L\in \{ (1-x), x s_1,  x s_2\}$. 
If $\Nld(1- a)=1$, then this dislocation is $(\psi_L, 1-a)$-large thus $(x,\s)$ verifies \eqref{eq:La}, which implies $L<a$. If $\Nld(1- a)=0$, then this dislocation is not $(\psi_L, 1-a)$-large, thus $\max(1-x, xs_1, xs_2)\geq a$, which implies $L\geq a$. Hence we conclude that $\Prob{L < a} = \Exp{\Nld(1- a)}$.
\end{proof}

To determine the law of the length of the longest chord in $\B$, we still need to calculate $ \Exp{\Nld(1- a)}$ explicitly. By the transformation in Section \ref{section:index}, which preserves the total number of large dislocations, we may regard $\Nld(1- a)$ as the number of $(\psi_L, 1-a)$-large dislocations of the homogeneous counterpart $\X$ of $\Theta_e$. By Lemma \ref{lem:Mart}, 
\begin{equation}
\Exp{\Nld(1- a)}= \Exp{\int_0^{\infty}\sum_{i=1}^{\infty} g_1 (1-a/X_i(r))\ind{X_i(r)>1-a} dr},
\end{equation}
where function $g_1: \R_+ \to \R_+$ is given by \eqref{eq:gBT}. By \eqref{eq:tagged}, the right-hand side equals
\begin{equation}\label{eq:PU}
\ind{a>\frac{1}{3}} \frac{2\sqrt{2}}{\sqrt{\pi}} \int_{\log{\frac{1}{2a}}}^{\log{\frac{1}{1-a}}} \e^{x}  \frac{2a {\e}^{x}-1}{ \sqrt{(1-a{\e}^{x}){\e}^{x}a}} dU(x).
\end{equation}
where, according to Lemma \ref{lem:xi} and \eqref{eq:phiBT}, the measure $dU$ is characterized by its Laplace transform
\begin{equation} 
\int_{0}^{\infty} \e^{-p x}dU(x)  =  \left( \int_{\Sd} \left(1- \sum_{i=1}^{\infty}s_i^{p+1}\right) \nu_e(d\s) \right)^{-1}= \frac{1}{2\sqrt{2}} \frac{\Gamma(p)}{\Gamma(p+1/2)}, \quad p> 0. 
\end{equation}
Noticing that
\begin{equation}
 \int_0^{+\infty} \e^{-p x} (1-\e^{-x})^{-1/2}dx = \Gamma(1/2)\frac{\Gamma(p) }{\Gamma(p +1/2) },\quad p> 0,
\end{equation}
we have   
\begin{equation}
dU(x)= \frac{1}{2\sqrt{2\pi}} (1-\e^{-x})^{-1/2}dx, \quad x\in [0,\infty). 
\end{equation}
Hence, rewriting \eqref{eq:PU}, we have
\begin{linenomath}\begin{align}
\Prob{L < a}& =\ind{a>\frac{1}{3}}\frac{1}{a\pi}\int_{\frac{1}{2}}^{\frac{a}{1-a}} \frac{2y-1}{\sqrt{(1-y)(y-a)}} dy\\
&=  \ind{a>\frac{1}{3}} \left( 6\pi^{-1}(\arctan(3^{-\frac{1}{2}})  -\arctan((1-2a)^{\frac{1}{2}}) ) -\frac{(3a-1)(1-2a)^{\frac{1}{2}}}{\pi a(1-a)}\right), \quad 0 < a <\frac{1}{2},
\end{align}\end{linenomath}
which is the formula (9) in \cite{aldous1994triangulating}.


\section{Random stable laminations of a disk}\label{section:stable}
In this section, we generalize our work about the Brownian triangulation to the stable laminations. Specifically, we will study the number of their large faces and find the laws of the lengths of their longest chords.

 A {\it (geodesic) lamination of the disk $\D$} is a closed subset of $\D$, which can be written as the union of a random collection of non-crossing chords of the circle $\partial \D$. In particular, a triangulation is a lamination. Conversely, it is easy to see that a lamination is a triangulation if and only if it is maximal for the inclusion relation among the laminations of $\D$. 

The {\it random stable lamination of the disk with parameter} $\beta \in (1,2]$ (or shortly {\it $\beta$-stable lamination}) is a random model of laminations which was introduced by Kortchemski \cite{kortchemski2014random}. For $\beta=2$, the $\beta$-stable lamination is the Brownian triangulation. Hence from now on, we consider the $\beta$-stable lamination with $\beta \in (1,2)$. It has been shown in \cite{kortchemski2014random} that the $\beta$-stable lamination is the distributional limit for the Hausdorff topology of certain families of random dissections of polygon $P_n$ when $n\to \infty$, which we do not describe here. But let us briefly present two other constructions in \cite{kortchemski2014random} that connect the $\beta$-stable lamination with the $\beta$-stable process and the $\beta$-stable tree. 

Let $X$ be a strictly $\beta$-stable spectrally positive L\'evy process, whose Laplace transform is 
\begin{equation}\Exp{\exp(-p X_t) } = \exp(p^{\beta}t), \quad \text{ for } t, p \geq 0.\end{equation}
 Let $X^{exc}$ be the normalized excursion of $X$ (see Section 2.1 of \cite{kortchemski2014random}). For $0\leq s<t\leq 1$, we define 
\begin{equation} s \overset{X^{exc}}{\simeq} t, \text{ if } t = \inf \{u>s; X^{exc}_u \leq X^{exc}_{s-}\}.\end{equation}
Then $L_{X^{exc}} := \bigcup_{ s\overset{X^{exc}}{\simeq} t} [{\rm e}^{ 2\pi is},{\rm e}^{ 2\pi it}]$ is the $\beta$-stable lamination. 

The random stable lamination can also be encoded by a random continuous function $(H^{exc}_t)_{t\in [0,1]}$. A way to define $H^{exc}$ is as follows. For every $t\geq 0$, set
\begin{equation}H_t :=\lim_{\epsilon \to 0} \frac{1}{\epsilon} \int_0^t \ind{X_s \leq \inf_{u\in [s,t]} X_u +\epsilon} ds ,\end{equation}
where the limit holds in probability. Replacing $H$ by its continuous modification, then we define $H^{exc}$ by the normalized excursion of $H$. More details about $H^{exc}$ can be found in Section 4 of \cite{kortchemski2014random} or Section 2.2 in \cite{miermont2003self1}.
Next, for $s, t\in [0,1]$, we define a relation $s \overset{H^{exc}}{\approx} t$ as follows. 
We first define $s \overset{H^{exc}}{\sim} t$ if and only if $H^{exc}(s)=H^{exc}(t)=\min_{r\in [s\wedge t,s\vee t]}H^{exc}(r)$ (which is the same as the relation to encode the Brownian triangulation), then we say $s \overset{H^{exc}}{\approx} t$, if $ s \overset{H^{exc}}{\sim} t$ and one of the two following properties is satisfied:
\begin{enumerate}
\item $\forall r \in (s\wedge t,s\vee t)$, $H^{exc}(r) > H^{exc}(s)=H^{exc}(t)$,
\item Let $cl_{H^{exc}}(s):=\{r| r\overset{H^{exc}}{\sim} s\}$, then $cl_{H^{exc}}(s) \subset [s\wedge t,s\vee t]$.
\end{enumerate}
According to Theorem 4.5 in \cite{kortchemski2014random}, almost surely
\begin{equation}L_{H^{exc}} := \bigcup_{s\overset{H^{exc}}{\approx} t} [{\rm e}^{ 2\pi is},{\rm e}^{ 2\pi it}] = L_{X^{exc}}.\end{equation}
Hence the $\beta$-stable lamination can be represented by either $L_{X^{exc}}$ or $L_{H^{exc}}$, and we will not distinguish them. 

Further, these representations imply the connection between the $\beta$-stable lamination and the fragmentation process 
\begin{equation}\Theta_{\beta}(t) = \{s\in (0,1): H^{exc}(s)>t \}, \quad t\geq 0.\end{equation}
According to Proposition 1 and Theorem 1 in \cite{miermont2003self1}, this process is a self-similar fragmentation with index $(1/\beta-1)<0$, no erosion and dislocation measure
\begin{equation}\label{eq:nubeta}
\nu_{\beta}(d\s)  = D_{\beta}\Exp{T_1; \frac{\Delta T_{[0,1]}}{T_1}\in d\s},
\end{equation}
where $(T_x)_{x\geq 0}$ is a $\beta$-stable subordinator, $\Delta T_{[0,1]} = (\Delta_1, \Delta_2 , \cdots)$ is the vector of jumps of $T$ before time $1$ reordered in the decreasing order, and $D_{\beta} = \frac{\beta^2\Gamma(2-1/\beta)}{\Gamma(2-\beta)}$.
The connection between the $\beta$-stable lamination and $\Theta_{\beta}$ is described by the following statement. 
\begin{prop}\label{prop_bilam}
Almost surely, there is a bijection between the faces in the $\beta$-stable lamination of $\D$ and the dislocations of the fragmentation $\Theta_{\beta}$. If a face corresponds to a dislocation labeled by $(x, \s)\in [0,1]\times \Sd$, then its vertices divide $\partial \D$ into arcs of lengths $2\pi(1-x, xs_1,xs_2, \cdots)$, and the edges of this face have lengths $(2\sin(\pi x),2\sin(\pi xs_1),2\sin(\pi xs_2) ,\cdots)$. 
\end{prop}
\begin{proof}
 According to Proposition 3.10 in \cite{kortchemski2014random}, there is a bijection between the faces of $L_{X^{exc}}$ and the jump time of $X^{exc}$. It is clear that the faces of $L_{H^{exc}}$ correspond to a subset of 
\begin{equation}\{cl_{H^{exc}}(u), u\in [0,1]: Card(cl_{H^{exc}}(u))\geq 3 \}.\end{equation}
Finally, by Proposition 4.4 in \cite{kortchemski2014random}, the latter set corresponds to a subset of the jump time of $X^{exc}$. Hence we have the bijection. The second assertion is plain by the geometry.
\end{proof}

\subsection{Large faces in the stable laminations}
Thanks to Proposition \ref{prop_bilam}, we can study the number of large faces in the $\beta$-stable lamination, $\beta\in (1,2)$. We would like to find a result of type Theorem \ref{thm:Brownian}. However, almost surely every face in the $\beta$-stable lamination has infinitely many sides, hence the shortest edge of a face is meaningless. On the other side, we find that the face corresponding to a dislocation labeled by $(x, \s)\in [0,1]\times \Sd$ has area $\frac{1}{2}\sin(2 \pi x) + \frac{1}{2} \sum_{i\in \N} \sin(2 \pi x s_i)$. If we want to estimate the number of faces of large area, then we have to study for every $x\in [0,1]$, the asymptotic behavior as $\epsilon \to 0$ of   
$$ f_{\epsilon}(x) = \nu_{\beta} \left(\s \in \Sd ~:~\frac{1}{2}\sin(2 \pi x)+ \frac{1}{2} \sum_{i\in \N} \sin(2 \pi x s_i) > \epsilon \right),$$
which seems rather difficult. Therefore, let us define alternatively the large faces in the following way. 

For each face, we consider the corresponding {\bf minor} arcs of its edges. For $\epsilon>0$, we define a face to be {\it $\epsilon$-large} if at least two of those arcs are longer than $\epsilon$, and the total length of the rest arcs is greater than $\epsilon$. This definition should be meaningful. In the Brownian triangulation case ($\beta=2$), the triangles with the shortest edges longer than $\epsilon$ are exactly the so-defined $\epsilon'$-large faces with $\epsilon'= 2 \sin(\pi \epsilon)$. We have the following result for the number of $\epsilon$-large faces in the $\beta$-stable lamination.
\begin{thm}\label{thm:sl}
For $\beta\in (1,2)$, let $\Nld(\epsilon)$ be the number of $\epsilon$-large faces defined as above in the $\beta$-stable lamination, then
\begin{equation} \lime \epsilon \Nld(\epsilon) = \frac{2\pi(\beta-1)}{\Gamma(2-\beta)} \Exp{\min(T_1 -\Delta_1, \Delta_1)} \cvL,\end{equation} 
where $T_1$ is the value of the $\beta$-stable subordinator $T$ at time $1$, and $\Delta_1$ is the largest jump of $T$ before time $1$. 
 \end{thm}

Before tackling the proof, let us look at the value of $\Exp{\min(T_1 -\Delta_1, \Delta_1)}$. 
We refer to \cite{perman1993jumps} for the joint law of $(T_1, \Delta_1)$, in which it has been proved that the joint law of $(T_1, \Delta_1)$ has a density, although this density function is not explicitly given. Let us calculate $\Exp{\min(T_1 -\Delta_1, \Delta_1)}$ by using the approach in \cite{perman1993jumps}. By the L\'evy-It\^o decomposition, we find that $T_1$ is the sum of the atoms of a Poisson random measure on $(0,+\infty)$ with intensity $C_{\beta}dr/r^{1+1/\beta}$, where $C_{\beta} = (\beta \Gamma(1-\beta^{-1}))^{-1}$. Thus for $y>0$, the probability that no atom has mass greater than $y$ is
\begin{equation}
\Prob{\Delta_1 \leq y}= \exp \left( -\int_y^{\infty} C_{\beta}dr/r^{1+1/\beta}\right) = \exp \left(-\frac{1}{\Gamma(1-1/\beta)} y^{-1/\beta} \right).
\end{equation}
By the restriction property of Poisson point processes, we see that the conditional distribution of $T_1 -\Delta_1$ given $\Delta_1=y$ is a subordinator with L\'evy measure $C_{\beta}\ind{r\leq y}dr/r^{1+1/\beta}$. Write $\tilde{T}^y_1$ for this subordinator, which is characterized by its Laplace exponent 
\begin{equation}
\Exp{\exp(-p \tilde{T}^y_1)} = \exp \left( -\int_0^y C_{\beta}/r^{1+1/\beta} (1-e^{-p r}) dr\right), \quad p \geq 0,
\end{equation}
then we have
 \begin{equation}
  \Exp{\min(T_1 -\Delta_1,\Delta_1 )} = \int_0^{\infty} \Exp{ \min(\tilde{T}^y_1, y)}  \exp \left(-\frac{1}{\Gamma(1-1/\beta)} y^{-1/\beta}\right) C_{\beta}y^{-(1+1/\beta)} dy.
 \end{equation}

\begin{proof}[Proof of Theorem \ref{thm:sl}]
Define $\psi_*: [0,1] \times \Sd\to [0,\infty)$ by $\psi_*(x,\s) = \min (s_1, 1-s_1) x$ and let $N^{\psi_*}(\epsilon)$ be the number of all $(\psi_*, \epsilon)$-large dislocations. Let us compare $\Nld^{\psi_*}(\frac{\epsilon}{2\pi})$ with $\Nld(\epsilon)$. 
For a face associated with a dislocation marked by $\markd$, the corresponding shorter arcs of the edges have lengths 
\begin{equation} \min(2\pi (1-x),2\pi x ),~ \min(2\pi (1-xs_1),2\pi xs_1 ),~ 2\pi xs_2=\min(2\pi (1-xs_2),2\pi xs_2 ),~ \cdots \end{equation}
When $x<\frac{1}{2}$, the two longest arcs have lengths $2\pi x$ and $2\pi xs_1$. Thus if a dislocation is $(\psi_*, \frac{\epsilon}{2\pi})$-large and its mark $\markd$ verifies $x<\frac{1}{2}$, then its corresponding face is $\epsilon$-large. 
As before, let $\bar{N}^{\psi_*}(\epsilon)$ be the number of $(\psi_*, \epsilon)$-large dislocations of fragments of masses greater than $\frac{1}{2}$.
Then 
\begin{equation}\label{eq:sl1}
\Nld^{\psi_*}(\frac{\epsilon}{2\pi})- \bar{N}^{\psi_*}(\frac{\epsilon}{2\pi}) \leq \Nld(\epsilon).
\end{equation}
On the other hand, if a dislocation with mark $\markd$ is not $(\psi_*, \frac{\epsilon}{2\pi})$-large, then $xs_1<\frac{\epsilon}{2\pi}$ or $x(1- s_1) < \frac{\epsilon}{2\pi}$. If $xs_1<\frac{\epsilon}{2\pi}$, then, as $2\pi xs_i\leq 2\pi xs_1 < \epsilon$ for $i\geq 2$, at most one arc is longer than $\epsilon$, thus the face is not $\epsilon$-large; if $x(1- s_1) < \frac{\epsilon}{2\pi}$, noticing that $2\pi x(1- s_1)$ is the total length of all the arcs except the two with lengths $\min(2\pi xs_1, 2\pi -2\pi xs_1)$ and $\min(2\pi x, 2\pi -2\pi x)$, we find the face not $\epsilon$-large. Hence 
\begin{equation}\label{eq:sl2}
\Nld(\epsilon) \leq \Nld^{\psi_*}(\frac{\epsilon}{2\pi}). 
\end{equation}

Next we study $\Nld^{\psi_*}(\epsilon)$ and $\bar{N}^{\psi_*}(\epsilon)$. By Section 4.4 in \cite{haasmiermont04genealogy}, $\exists C>0$ such that
\begin{equation}\nu_{\beta}(\s \in \Sd:\min(s_1, 1-s_1)>u)\leq  \nu_{\beta}(\s \in \Sd: 1-s_1>u)\sim C u^{-(1-1/\beta)},\quad  u\to 0^+.\end{equation}
Since $\psi_*$ has the form \eqref{eq:psi} and $(H1)$ holds as $1-1/\beta<1$, by Theorem \ref{thm:large} we have
\begin{equation}\label{eq:sl3}
 \lime \epsilon N^{\psi_*}(\frac{\epsilon}{2\pi}) = 2\pi\frac{1}{m}\int_0^{1}\nu_{\beta} \left(\s \in \Sd:\min (s_1, 1-s_1)> u\right) du \cvL,
\end{equation} 
where $m = \int_{\Sd} \sum_{i=1}^{\infty} s_i \log (s_i^{-1}) \nu_{\beta}(ds) $. 
Using Lemma \ref{lem:small}, we get
\begin{equation}\label{eq:sl4}
 \lime \epsilon \bar{N}^{\psi_*}(\epsilon)= 0 \cvL.
\end{equation} 
Then, combining \eqref{eq:sl1}, \eqref{eq:sl2}, \eqref{eq:sl3} and \eqref{eq:sl4}, we have 
\begin{equation}
\lime \epsilon \Nld(\epsilon)  = 2\pi\frac{1}{m}\int_0^{1}\nu_{\beta}\left(\s \in \Sd:\min (s_1, 1-s_1)> u\right) du \cvL.
\end{equation} 
To compute the value of limit, let us introduce the size biased picked jump $\Delta^*$ among $(\Delta_i)_{i\in \N}$, the jumps of a $\beta$-stable subordinator $T$ before time $1$. The law of $\Delta^*$ conditionally on $(\Delta_i)_{i\in \N}$ is given by \begin{equation}
\Probcond{\Delta^* = \Delta_j}{(\Delta_i)_{i\in \N}  } = \frac{\Delta_j}{T_1}, \quad \forall j\in \N.
\end{equation} 
Then, on the one hand, we deduce from \eqref{eq:nubeta} that
\begin{linenomath}\begin{align}
m =  \int_{\Sd} \sum_{i=1}^{\infty} s_i \log (s_i^{-1}) \nu_{\beta}(d\s) = D_{\beta} \Exp{T_1 \sum_{i=1}^{\infty} \frac{\Delta_i}{T_1} \log (\frac{T_1}{\Delta_i})}=\Exp{T_1 \log (\frac{T_1}{\Delta^*})}.
\end{align}\end{linenomath}
According to Lemma 1 in \cite{miermont2003self1}, the joint law of $(\Delta^*, T_1)$ has density:
\begin{equation}
\frac{C_{\beta} q_1(s-y)}{s y^{1/\beta}}dyds , \quad (y,s)\in [0,\infty)^2,
\end{equation} 
where the constant $C_{\beta} = \frac{1}{\beta \Gamma(1-1/\beta)}$, and $q_1$ is the density function of $T_1$. Then 
\begin{linenomath}\begin{align}
m &= D_{\beta} C_{\beta} \int_{[0,\infty)^2} \log(s/y) q_1(s-y)y^{-\frac{1}{\beta}} dy ds \\ 
& =  D_{\beta} C_{\beta} \int_0^{\infty} q_1(u)du \int_u^{\infty} \log(\frac{s}{s-u})(s-u)^{-\frac{1}{\beta}} ds \\
& = D_{\beta} C_{\beta} \frac{\beta}{\beta-1} \frac{\pi}{\sin(\pi/\beta)} \Exp{T_1^{1-1/\beta}} .
\end{align}\end{linenomath}
It is well-known that $\Exp{T_1^{1-1/\beta}} = \frac{\Gamma(2-\beta)}{\Gamma(1/\beta)}$ and $\Gamma(1- 1/\beta)\Gamma(1/\beta) = \frac{\pi}{\sin(\pi /\beta)}$, so we find that 
$$m = D_{\beta} \frac{\Gamma(2-\beta)}{\beta-1}.$$
On the other hand, by \eqref{eq:nubeta} and Fubini's Theorem, we have
\begin{equation}
\int_0^{1}\nu_{\beta}(\s \in \Sd:\min (s_1, 1-s_1)> u) du =  D_{\beta}\Exp{T_1 \int_0^{1}\ind{\min (\frac{T_1-\Delta_1}{T_1}, \frac{\Delta_1}{T_1})> u}du }
 = D_{\beta}\Exp{\min(T_1-\Delta_1, \Delta_1)}.
\end{equation}
Then the claim follows. We note that $\Exp{\min(T_1-\Delta_1,\Delta_1 )}$ is indeed finite, since 
\begin{equation}
\Exp{\min(T_1-\Delta_1,\Delta_1 )} \leq \Exp{T_1 -\Delta_1}\leq \Exp{T_1 -\Delta^*}
\end{equation}  
and we can check by using the joint law of $(\Delta^*, T_1)$ that $\Exp{T_1 -\Delta^*}< \infty$.  
\end{proof}

\subsection{Length of the longest chord}
For $\beta\in (1,2)$, we will find the law of the length of the longest chord in the $\beta$-stable lamination in the same way as in the Brownian triangulation. Let $2\pi L$ be the length of the minor arc corresponding to the longest chord in the $\beta$-stable lamination. In the $\beta$-stable lamination, almost surely no chord passes through the origin. Thus we still call the unique face that contains the origin the {\it centroid}, and the longest chord is still an edge of the centroid. Hence using the bijection obtained by Proposition \ref{prop_bilam} and the same arguments as in the proof of Lemma \ref{lem:longest}, we can prove that 
\begin{equation}
\Prob{L < a} = \Exp{\Nld(1-a)}, \quad 0<a<\frac{1}{2},
\end{equation}
where $N(1-a)$ is the number of $(\psi_L, 1-a)$-large dislocations in $\Theta_{\beta}$ as in Lemma \ref{lem:longest}.
This equation allows us to find the law of the length of the longest chord in the $\beta$-stable lamination. 

\begin{prop}\label{prop:arcsl}
For $\beta\in (1,2)$, let $2\pi L$ be the length of the minor arc associated with the longest chord in the $\beta$-stable lamination, then $L$ has distribution function  
\begin{linenomath}\begin{align}
\Prob{L < a}& =D_{\beta}C_{\beta} \Exp{T_1 \int_{\frac{\Delta_1}{aT_1} \vee 1}^{\frac{1}{1-a}} (1-x^{-1})^{-1/\beta} dx } , \quad 0<a< \frac{1}{2}, 
 \end{align}\end{linenomath}
where $D_{\beta}= \frac{\beta^2\Gamma(2-\beta^{-1})}{\Gamma(2-\beta)}$, $C_{\beta}=  \frac{1}{\beta \Gamma(1-\beta^{-1})}$. 
\end{prop}
\begin{proof}[Proof of Proposition \ref{prop:arcsl}]
For $a\in (0, \frac{1}{2})$, as we have argued above, $\Prob{L \leq a} = \Exp{\Nld(1-a)}$. Let us calculate $\Exp{\Nld(1-a)}$. Using the transformation in Section \ref{section:index}, we may regard $\Nld(1-a)$ as the number of $(\psi_L, 1-a)$-large dislocations in $\X$, the homogeneous counterpart of $\Theta_{\beta}$. Recall that $\psi_L(x, \s) = \min(x, 1-xs_1)$, $\markd$, it follows from Lemma \ref{lem:Mart} that
\begin{equation} \Exp{\Nld(1-a)} =  \Exp{\int_0^{\infty}\sum_{i=1}^{\infty} \nu_{\beta} \left(\s\in \Sd: 1-s_1> 1-\frac{a}{X_i(r)}\right) \ind{X_i(r)>1-a} dr}.\end{equation}  
Using \eqref{eq:tagged}, \eqref{eq:nubeta} and Fubini's Theorem yields
\begin{equation}\label{eq:la1}
\Exp{\Nld(1-a)} = \int_0^{\log{\frac{1}{1-a}}} \e^{x} \int_{\Sd} \ind{1-s_1 > 1-a \e^{x}} \nu_{\beta}(d\s)dU(x) = D_{\beta} \Exp{T_1 \int_{(\log\frac{\Delta_1}{aT_1}) \vee 0}^{\log{\frac{1}{1-a}}}\e^{x}dU(x) },
\end{equation}
where $U$ is the potential measure of the subordinator associated with $\X$ as in Lemma \ref{lem:xi}. From Lemma \ref{lem:xi} and Equation (10) in \cite{miermont2003self1}, the Laplace transform of $U$ is 
\begin{equation}\int_{0}^{\infty}e^{-p x}dU(x)  =\left( \int_{\Sd}\left(1- \sum^{\infty}_{i=1} s_i^{p+1}\right) \nu_{\beta}(d\s) \right)^{-1}=\left( \beta \frac{\Gamma(p + 1 - 1/\beta)}{\Gamma(p)} \right)^{-1}, \quad p> 0.
\end{equation} 
Noticing that
\begin{equation}
 \int_0^{+\infty} \e^{-p x} (1-\e^{-x})^{-1/\beta} dx = \Gamma(1-1/\beta)\frac{\Gamma(p) }{\Gamma(p +1 -1/\beta) }, \quad p> 0,
\end{equation}
we find the density of $dU$:
\begin{equation}
dU(x)= C_{\beta} (1-\e^{-x})^{-1/\beta}dx, \quad x\geq 0, 
\end{equation}
where $ C_{\beta} = (\beta \Gamma(1-1/\beta))^{-1}$. Rewriting \eqref{eq:la1}, we have
\begin{linenomath}\begin{align}
\Prob{L < a} = \Exp{\Nld(1-a)}& = D_{\beta}C_{\beta} \Exp{T_1 \int_{\frac{\Delta_1}{aT_1} \vee 1}^{\frac{1}{1-a}} \left(1-x^{-1}\right)^{-1/\beta} dx } .
\end{align}\end{linenomath}
\end{proof}


\section{Proofs of the technical statements} \label{section:technical}
In the section we prove Lemma \ref{lem:m2} and  Lemma \ref{lem:small}.
\begin{proof}[Proof of Lemma \ref{lem:m2}]
Given $V$, a uniform random variable in $(0,1)$ independent of $\X$, we obtain the fragment tagged by $V$ in the same way as in Section \ref{section:tagged}, and denote its mass by $\chi$. Set
\begin{equation} \tilde{A}_{t}(\epsilon) =\int_0^{t} \chi(s)^{-1} f_1(\epsilon^{-\frac{1}{b}}\chi(s)) ds, \quad t\geq 0. \end{equation} 
Note that $f_1(0)=0$ and by convention $0^{-1}\times 0=0$. Let $\tilde{A}(\epsilon):= \lim_{t\to \infty}\tilde{A}_{t}(\epsilon)$. Given another random variable $V'$, uniform in $(0,1)$ and independent of $V$ and $\X$, we define $\chi'$, $\tilde{A}'_{t}(\epsilon)$ and $\tilde{A}'(\epsilon)$ in the same way. Using \eqref{eq:tagged} yields 
\begin{equation}\Exp{\tilde{A}_t(\epsilon)} = \Exp{\tilde{A}'_t(\epsilon)} = \Exp{\int_0^t \sum_{i=1}^{\infty} f_1(\epsilon^{-\frac{1}{b}} X_i(r))dr}.\end{equation}
Since $(V,V')$ has uniform distribution in $[0,1]^2$, independent of $\F_t$, we deduce the distribution of $(\chi(t),\chi'(t))$ conditionally on $\F_t$: 
\begin{equation}\Probcond{(\chi(t),\chi'(t)) = (X_i(t),X_j(t))}{\F_t} = X_i(t)X_j(t), \quad \forall i,j\in \N.\end{equation}
By standard calculation, there is 
\begin{equation}\Exp{\tilde{A}_t(\epsilon)\tilde{A}_t'(\epsilon)} =\Exp{\left(\int_0^t \sum_{i=1}^{\infty} f_1(\epsilon^{-\frac{1}{b}} X_i(r))dr \right)^2}.\end{equation}
Letting $t \to +\infty$ and using the monotone convergence theorem, we obtain 
\begin{equation}\Exp{\tilde{A}(\epsilon)}=\Exp{\Ald(\epsilon)} \quad \text{and}\quad \Exp{\tilde{A}(\epsilon)\tilde{A}'(\epsilon)}=\Exp{\Ald(\epsilon)^2}.\end{equation}
Let $T$ be the first instant when $V$ and $V'$ belong to two different intervals, with respective lengths $\chi(T)$ and $\chi'(T)$. Set 
\begin{equation}S:= \tilde{A}_T(\epsilon), \quad R:= \int_T^{\infty} \chi(t)^{-1}f_1(\epsilon^{-\frac{1}{b}}\chi(t)) dt, \quad R':= \int_T^{\infty} \chi'(t)^{-1}f_1(\epsilon^{-\frac{1}{b}}\chi'(t)) dt .\end{equation}
For conciseness we did not indicate their dependence on $\epsilon$. 
Then $\tilde{A}(\epsilon)=  S+R$ and $\tilde{A}'(\epsilon)= S+R'$. Hence 
\begin{equation}\Exp{\Ald(\epsilon)^2} = \Exp{S^2 + S(R +R')+RR'}.\end{equation} 
We will calculate each term.
Let us begin with $\Exp{S(R+R')}$. 
By Markov property, 
\begin{equation}
 \Expcond{S(R+ R')}{S,\chi(T),\chi'(T) } 
=S \left( \chi(T)^{-1}\left. \Exp{\Ald(x^{-b}\epsilon)}\right|_{x=\chi(T) } + \chi'(T)^{-1}\left. \Exp{\Ald((x')^{-b}\epsilon)}\right|_{x' =\chi'(T) } \right).
\end{equation}
As $\Ald(\epsilon)= 0$ for $\epsilon> |\varphi|$, from Lemma \ref{lem:m1} it is clear that there exists $\bar{c}_A>0$ such that for every $\epsilon>0$,
\begin{equation}\label{eq:m1bar}
\epsilon^{\frac{1}{b}} \Exp{\Ald(\epsilon)}\leq \bar{c}_A.
\end{equation}
Thus
\begin{equation}\label{eq:SBB'}
  \epsilon^{\frac{2}{b}} \Exp{S(R+R')}\leq 2\bar{c}_A  \epsilon^{\frac{1}{b}} \Exp{S}.
\end{equation}

We now deal with $\Exp{RR'}$. Let $\bar\chi$ and $\bar{\chi}'$ be two independent copies of $\chi$. By the branching property and the Markov property, we have that
\begin{linenomath}\begin{align}
& \Expcond{RR' }{\chi(T), \chi'(T)}\\
& = \left. \Exp{\int_0^{\infty} (x \bar{\chi}(s))^{-1}f_1(x\epsilon^{-\frac{1}{b}}\bar{\chi}(s))ds \int_0^{\infty} (x'\bar{\chi}'(r))^{-1}f_1(x'\epsilon^{-\frac{1}{b}}\bar{\chi}'(r))dr} \right|_{x =\chi(T), x'=\chi'(T)}\\
& =  \chi(T)^{-1}\chi'(T)^{-1} \left.\Exp{\Ald(x^{-b}\epsilon)}\right|_{x =\chi(T) } \left. \Exp{\Ald((x')^{-b}\epsilon)}\right|_{x' =\chi'(T) }.
\end{align}\end{linenomath}
It then follows from \eqref{eq:m1bar} that for every $\epsilon>0$,
\begin{equation}\epsilon^{\frac{2}{b}} \Expcond{RR' }{\chi(T), \chi'(T)} \leq (\bar{c}_A)^2. \end{equation}
We deduce from Lemma \ref{lem:m1} that 
\begin{equation}\lim_{\epsilon \to 0} \epsilon^{\frac{2}{b}}\Expcond{RR' }{\chi(T), \chi'(T)} = \left(\frac{1}{m} \int_0^{\infty} g(u^b)du \right)^2. \end{equation}
Taking expectation in the last limit and using the dominated convergence theorem, we have that
\begin{equation}\label{eq:BB'}
\lim_{\epsilon\to 0} \epsilon^{\frac{2}{b}}\Exp{RR'} = \left(\frac{1}{m} \int_0^{\infty} g(u^b)du \right)^2.
\end{equation}

We next calculate $\Exp{S^2}$. It is clear that
\begin{linenomath}\begin{align}
  \Exp{S^2} &=  2 \int_{(0,\infty)^2} \Exp{ \ind{s<t} \ind{s<T} \ind{t<T} \chi(s)^{-1} \chi(t)^{-1}f_1(\epsilon^{-\frac{1}{b}}\chi(s))f_1(\epsilon^{-\frac{1}{b}}\chi(t)) } ds dt\\
 &\leq  2 \int_{(0,\infty)^2} \Exp{ \ind{s<t} \ind{s<T} \chi(s)^{-1} \chi(t)^{-1}f_1(\epsilon^{-\frac{1}{b}}\chi(s))f_1(\epsilon^{-\frac{1}{b}}\chi(t)) } ds dt. 
\end{align}\end{linenomath}
Let $\bar{\chi}$ be an independent copy of $\chi$. By the Markov property, the last quantity equals
\begin{equation} 2 \int_{0}^{\infty} ds \Exp{ \ind{s<T} \chi(s)^{-1} f_1(\epsilon^{-\frac{1}{b}}\chi(s)) \left. \Exp{\int_{0}^{\infty} x^{-1}\bar{\chi}(u)^{-1}f_1(\epsilon^{-\frac{1}{b}} x \bar{\chi}(u)) du } \right|_{x= \chi(s)}} .\end{equation}
Multiplying by $\epsilon^{\frac{2}{b}}$ and using \eqref{eq:m1bar}, we obtain that
\begin{equation}\label{eq:S2}
 \epsilon^{\frac{2}{b}} \Exp{S^2} \leq 2 \epsilon^{\frac{1}{b}} \int_{0}^{\infty} ds \Exp{ \ind{s<T} \chi(s)^{-1}f_1(\epsilon^{-\frac{1}{b}}\chi(s)) \bar{c}_A } = 2\bar{c}_A  \epsilon^{\frac{1}{b}}\Exp{S}.
\end{equation}

We finally look at $\Exp{S}$. For $t> 0$, write $I_{n(t)}(t)$ for the interval fragment containing $V$ at time $t$ as in Section \ref{section:tagged}, thus $|I_{n(t)}(t)|=\chi(t)$. Since $V'$ is independent of $I_{n(t)}(t)$, 
 \begin{equation}
\Expcond{\chi(t)^{-1}f_1(\epsilon^{-\frac{1}{b}}\chi(t))\ind{t\leq T}}{I_{n(t)}(t)} = \chi(t)^{-1}f_1(\epsilon^{-\frac{1}{b}}\chi(t)) \Expcond{\ind{V'\in I_{n(t)}(t)}}{I_{n(t)}(t)}  = f_1(\epsilon^{-\frac{1}{b}}\chi(t)).\end{equation}
Therefore, we have
\begin{linenomath}\begin{align}&\Exp{S} =\int_0^{\infty} \Exp{ \Expcond{\chi(t)^{-1}f_1(\epsilon^{-\frac{1}{b}}\chi(t))\ind{t\leq T}}{I_{n(t)}(t)} }dt = \int_0^{\infty}\Exp{f_1(\epsilon^{-\frac{1}{b}}\chi(t))} dt.  \end{align}\end{linenomath} 
Multiplying by $\epsilon^{\frac{1}{b}}$ and using \eqref{eq:cbarf} and then Lemma \ref{lem:xi}, we have that 
\begin{linenomath}\begin{align}
\epsilon^{\frac{1}{b}}\int_0^{\infty}\Exp{f_1(\epsilon^{-\frac{1}{b}}\chi(t))} dt  \leq \epsilon^{\frac{1-ab}{b}} \bar{c} \int_0^{\infty} \Exp{\e^{-ab\xi(t)}} dt  = \bar{c} \epsilon^{\frac{1-ab}{b}} \frac{1}{\Phi(ab)}.
 \end{align}\end{linenomath} 
Since $1>ab$, the right-hand side tends to $0$ as $\epsilon \to 0$. Hence
\begin{equation}\label{eq:S}
\lim_{\epsilon \to 0} \epsilon^{\frac{1}{b}} \Exp{S}=\lim_{\epsilon \to 0} \epsilon^{\frac{1}{b}} \Exp{\int_0^{\infty}f_1(\epsilon^{-\frac{1}{b}}\chi(t))dt} = 0.
\end{equation}
Combining \eqref{eq:SBB'},\eqref{eq:BB'}, \eqref{eq:S2} and \eqref{eq:S}, we prove this lemma.
\end{proof}


\begin{proof}[Proof of Lemma \ref{lem:small}]
Define the function $\bar{\psi}$ by $\bar{\psi}(x, \s):= \psi(x,\s)\ind{x>\frac{1}{2}}$, $(x,\s)\in [0,1]\times \Sd$. Then $\bar{N}(\epsilon)$ is the number of $(\bar{\psi}, \epsilon)$-large dislocations, and the random variable defined as in \eqref{eq:Ainf} is
\begin{equation} \bar{A}(\epsilon) :=  \int_0^{\infty} \sum_{i=1}^{\infty} f_1 (X_i(r)\epsilon^{-\frac{1}{b}}) \ind{X_i(r)>\frac{1}{2}} dr.\end{equation}
Applying \eqref{eq:tagged}, we have
\begin{linenomath}\begin{align}
 \Exp{\bar{A}(\epsilon)} =  \int_0^{\infty} \Exp{\chi(r)^{-1} f_1 (\chi(r)\epsilon^{-\frac{1}{b}}) \ind{\chi(r)>\frac{1}{2}}} dr  \leq  2 \Exp{\int_0^{\infty} f_1(\chi(r)\epsilon^{-\frac{1}{b}}) dr } .
\end{align}\end{linenomath}
Multiplying by $\epsilon^{\frac{1}{b}}$ and using \eqref{eq:S}, we have
\begin{equation}\label{eq:barA}
\lim_{\epsilon \to 0}\epsilon^{\frac{1}{b}} \Exp{\bar{A}(\epsilon)}=0.
\end{equation}
Next we study $\Exp{\bar{A}(\epsilon)^2}$. Using the same notations as in the proof of Lemma \ref{lem:m2}, we set
\begin{linenomath}\begin{align}\bar{S}& := \int_0^{T} \chi(t)^{-1}f_1(\epsilon^{-\frac{1}{b}}\chi(t))\ind{\chi(t)>\frac{1}{2}} dt,\\
\bar{R}& := \int_T^{\infty} \chi(t)^{-1}f_1(\epsilon^{-\frac{1}{b}}\chi(t))\ind{\chi(t)>\frac{1}{2}} dt,\\
\bar{R}'&:= \int_T^{\infty} \chi'(t)^{-1}f_1(\epsilon^{-\frac{1}{b}}\chi'(t))\ind{\chi'(t)>\frac{1}{2}} dt.\end{align}\end{linenomath}
Thus similarly there is
\begin{equation}\Exp{\bar{A}(\epsilon)^2} = \Exp{\bar{S}^2 + \bar{S}(\bar{R} +\bar{R}')+\bar{R}\bar{R}'}.\end{equation}
On the one hand, letting $\bar\chi$ and $\bar{\chi}'$ be two independent copies of $\chi$, we see from the branching property and the Markov property that
\begin{linenomath}\begin{align}
& \Expcond{\bar{R}\bar{R}' }{\chi(T), \chi'(T)}\\
& = \left. \Exp{\int_0^{\infty} (x \bar{\chi}(s))^{-1}\ind{x \bar{\chi}(s)>\frac{1}{2}} f_1(x\epsilon^{-\frac{1}{b}}\bar{\chi}(s))ds \int_0^{\infty} (x'\bar{\chi}'(r))^{-1}\ind{x' \bar{\chi}'(r)>\frac{1}{2}}f_1(x'\epsilon^{-\frac{1}{b}}\bar{\chi}'(r))dr} \right|_{x =\chi(T), x'=\chi'(T)}\\
& \leq \left. \Exp{\int_0^{\infty} (x \bar{\chi}(s))^{-1}\ind{ \bar{\chi}(s)>\frac{1}{2}} f_1(x\epsilon^{-\frac{1}{b}}\bar{\chi}(s))ds \int_0^{\infty} (x'\bar{\chi}'(r))^{-1}\ind{ \bar{\chi}'(r)>\frac{1}{2}}f_1(x'\epsilon^{-\frac{1}{b}}\bar{\chi}'(r))dr} \right|_{x =\chi(T), x'=\chi'(T)}\\
& =  \chi(T)^{-1}\chi'(T)^{-1} \left.\Exp{\bar{A}(x^{-b}\epsilon)}\right|_{x =\chi(T) } \left. \Exp{\bar{A}((x')^{-b}\epsilon)}\right|_{x' =\chi'(T) }.
\end{align}\end{linenomath}
Then it follows from \eqref{eq:barA} that
\begin{equation}\lim_{\epsilon\to 0} \epsilon^{\frac{2}{b}}\Expcond{\bar{R}\bar{R}' }{\chi(T), \chi'(T)} =0. \end{equation}
Taking expectation and using the dominated convergence theorem, we see that 
\begin{equation}\label{eq:bar1}
\lim_{\epsilon\to 0} \epsilon^{\frac{2}{b}}\Exp{\bar{R}\bar{R}'} =0.
\end{equation}
On the other hand, we observe from \eqref{eq:SBB'} and \eqref{eq:S2} that
\begin{equation} \epsilon^{\frac{2}{b}} \Exp{\bar{S}^2 + \bar{S}(\bar{R} +\bar{R}')}\leq \epsilon^{\frac{2}{b}}  \Exp{S^2 + S(R +R')} \leq 4\bar{c}_A  \epsilon^{\frac{1}{b}} \Exp{S},\end{equation}
then it follows from \eqref{eq:S} that
\begin{equation}\label{eq:bar2}
\lim_{\epsilon\to 0}\epsilon^{\frac{2}{b}} \Exp{\bar{S}^2 + \bar{S}(\bar{R} +\bar{R}')} =0.
\end{equation}
So we conclude 
\begin{equation}\lim_{\epsilon \to 0}\epsilon^{\frac{2}{b}} \Exp{\bar{A}(\epsilon)^2}=\lim_{\epsilon\to 0}\epsilon^{\frac{2}{b}} \Exp{\bar{S}^2 + \bar{S}(\bar{R} +\bar{R}') + \bar{R}\bar{R}'} = 0.\end{equation}
The claim thus follows from Corollary \ref{cor_NA}. 
\end{proof}


\bibliographystyle{amsplain}
\bibliography{Quan_Shi.bib}

\providecommand{\bysame}{\leavevmode\hbox to3em{\hrulefill}\thinspace}
\providecommand{\MR}{\relax\ifhmode\unskip\space\fi MR }
\providecommand{\MRhref}[2]{%
  \href{http://www.ams.org/mathscinet-getitem?mr=#1}{#2}
}
\providecommand{\href}[2]{#2}
\begin{thebibliography}{10}

\bibitem{aldous1994recursive}
David Aldous, \emph{Recursive self-similarity for random trees, random
  triangulations and {B}rownian excursion}, Ann. Probab. \textbf{22} (1994),
  no.~2, 527--545. \MR{1288122 (95i:60007)}

\bibitem{aldous1994triangulating}
\bysame, \emph{Triangulating the circle, at random}, Amer. Math. Monthly
  \textbf{101} (1994), no.~3, 223--233. \MR{1264002 (95b:60011)}

\bibitem{berestycki2002ranked}
Julien Berestycki, \emph{Ranked fragmentations}, ESAIM Probab. Statist.
  \textbf{6} (2002), 157--175 (electronic). \MR{1943145 (2004d:60196)}

\bibitem{bertoin1996Levy}
Jean Bertoin, \emph{L\'evy processes}, Cambridge Tracts in Mathematics, vol.
  121, Cambridge University Press, Cambridge, 1996. \MR{1406564 (98e:60117)}

\bibitem{bertoin1999subordinators}
\bysame, \emph{Subordinators: examples and applications}, Lectures on
  probability theory and statistics ({S}aint-{F}lour, 1997), Lecture Notes in
  Math., vol. 1717, Springer, Berlin, 1999, pp.~1--91. \MR{1746300
  (2002a:60001)}

\bibitem{bertoin2001homogeneous}
\bysame, \emph{Homogeneous fragmentation processes}, Probab. Theory Related
  Fields \textbf{121} (2001), no.~3, 301--318. \MR{1867425 (2002j:60127)}

\bibitem{bertoin2002self}
\bysame, \emph{Self-similar fragmentations}, Ann. Inst. H. Poincar\'e Probab.
  Statist. \textbf{38} (2002), no.~3, 319--340. \MR{1899456 (2003h:60109)}

\bibitem{bertoin2003asymptotic}
\bysame, \emph{The asymptotic behavior of fragmentation processes}, J. Eur.
  Math. Soc. (JEMS) \textbf{5} (2003), no.~4, 395--416. \MR{2017852
  (2005d:60115)}

\bibitem{bertoin2012area}
\bysame, \emph{The area of a self-similar fragmentation}, ALEA Lat. Am. J.
  Probab. Math. Stat. \textbf{9} (2012), 53--66. \MR{2880708}

\bibitem{bertoin2005energy}
Jean Bertoin and Servet Mart{\'{\i}}nez, \emph{Fragmentation energy}, Adv. in
  Appl. Probab. \textbf{37} (2005), no.~2, 553--570. \MR{2144567 (2006a:60154)}

\bibitem{curienquadtrees}
Nicolas Curien, \emph{Strong convergence of partial match queries in random
  quadtrees}, Combin. Probab. Comput. \textbf{21} (2012), no.~5, 683--694.
  \MR{2959860}

\bibitem{curien2014dissecting}
\bysame, \emph{Dissecting the circle, at random}, ESAIM: Proc. \textbf{44}
  (2014), 129--139.

\bibitem{curien2012random}
Nicolas Curien and Igor Kortchemski, \emph{Random non-crossing plane
  configurations: a conditioned {G}alton-{W}atson tree approach}, Random
  Structures Algorithms \textbf{45} (2014), no.~2, 236--260. \MR{3245291}

\bibitem{curienlegall2011random}
Nicolas Curien and Jean-Fran{\c{c}}ois Le~Gall, \emph{Random recursive
  triangulations of the disk via fragmentation theory}, Ann. Probab.
  \textbf{39} (2011), no.~6, 2224--2270. \MR{2932668}

\bibitem{curien2013hyperbolic}
Nicolas Curien and Wendelin Werner, \emph{The {M}arkovian hyperbolic
  triangulation}, J. Eur. Math. Soc. (JEMS) \textbf{15} (2013), no.~4,
  1309--1341. \MR{3055763}

\bibitem{Durrett}
Rick Durrett, \emph{Probability: theory and examples}, fourth ed., Cambridge
  Series in Statistical and Probabilistic Mathematics, Cambridge University
  Press, Cambridge, 2010. \MR{2722836 (2011e:60001)}

\bibitem{haasmiermont04genealogy}
B{\'e}n{\'e}dicte Haas and Gr{\'e}gory Miermont, \emph{The genealogy of
  self-similar fragmentations with negative index as a continuum random tree},
  Electron. J. Probab. \textbf{9} (2004), no. 4, 57--97 (electronic).
  \MR{2041829 (2004m:60086)}

\bibitem{janson2007area}
Svante Janson, \emph{Brownian excursion area, {W}right's constants in graph
  enumeration, and other {B}rownian areas}, Probab. Surv. \textbf{4} (2007),
  80--145. \MR{2318402 (2008d:60106)}

\bibitem{neininger2008fragtrees}
Svante Janson and Ralph Neininger, \emph{The size of random fragmentation
  trees}, Probab. Theory Related Fields \textbf{142} (2008), no.~3-4, 399--442.
  \MR{2438697 (2010b:60065)}

\bibitem{kortchemski2014random}
Igor Kortchemski, \emph{Random stable laminations of the disk}, Ann. Probab.
  \textbf{42} (2014), no.~2, 725--759. \MR{3178472}

\bibitem{legall2008scaling}
Jean-Fran{\c{c}}ois Le~Gall and Fr{\'e}d{\'e}ric Paulin, \emph{Scaling limits
  of bipartite planar maps are homeomorphic to the 2-sphere}, Geom. Funct.
  Anal. \textbf{18} (2008), no.~3, 893--918. \MR{2438999 (2010a:60030)}

\bibitem{miermont2003self1}
Gr{\'e}gory Miermont, \emph{Self-similar fragmentations derived from the stable
  tree. {I}. {S}plitting at heights}, Probab. Theory Related Fields
  \textbf{127} (2003), no.~3, 423--454. \MR{2018924 (2005m:60163)}

\bibitem{perman1993jumps}
Mihael Perman, \emph{Order statistics for jumps of normalised subordinators},
  Stochastic Process. Appl. \textbf{46} (1993), no.~2, 267--281. \MR{1226412
  (94g:60141)}

\bibitem{rogers2000diffusions}
L.~C.~G. Rogers and David Williams, \emph{Diffusions, {M}arkov processes, and
  martingales. {V}ol. 2}, Wiley Series in Probability and Mathematical
  Statistics: Probability and Mathematical Statistics, John Wiley \& Sons,
  Inc., New York, 1987, It{\^o} calculus. \MR{921238 (89k:60117)}

\end{thebibliography}

\end{document}